\documentclass[11pt]{amsart}
\pdfoutput=1

\usepackage{amsmath,amsthm,amsfonts,amssymb,amscd}
\usepackage{enumitem}
\usepackage{microtype}
\usepackage[pdftex,colorlinks,citecolor=black,linkcolor=black,urlcolor=black,bookmarks=false]{hyperref}
\usepackage{mathtools}
\usepackage[left=1in,right=1in,top=1in,bottom=1in]{geometry}
\usepackage{xspace}
\usepackage[foot]{amsaddr}
\usepackage{tikz}
\newif\ifconfversion
\confversionfalse

\newcommand{\mysubsection}[1]{\vspace{0pt}\subsection{#1}}
\newcommand{\mysection}[1]{\vspace{0pt}\section{#1}}

\newtheorem{theorem}{Theorem}[section]
\newtheorem{lemma}[theorem]{Lemma}

\newtheorem{corollary}[theorem]{Corollary}
\newtheorem{definition}[theorem]{Definition}
\newtheorem{proposition}[theorem]{Proposition}
\newtheorem{observation}[theorem]{Observation}

\theoremstyle{definition}

\newtheorem{example}[theorem]{Example}

\numberwithin{equation}{section}

\newcommand\C{{\mathbb{C}}}

\newcommand\F{{\mathbb{F}}}
\newcommand\Z{{\mathbb{Z}}}

\newcommand\N{{\mathbb{N}}}

\DeclareMathOperator{\GL}{GL}

\newcommand{\D}{\mathcal{D}}

\newcommand{\R}{\mathcal{R}}
\newcommand{\m}{\mathbf{m}}

\newcommand{\mult}[1]{M_{#1}}
\DeclareMathOperator{\Var}{Var}

\newcommand{\abs}[1]{\left\lvert#1\right\rvert}

\newcommand{\onto}{\twoheadrightarrow}

\DeclareMathOperator{\Rank}{rank}
\DeclareMathOperator{\Span}{span}

\newcommand{\coloneq}{\mathrel{\mathop:}\mkern-1.2mu=}
\DeclareMathOperator{\Aut}{Aut}

\DeclareMathOperator{\slicerank}{slice-rank}
\DeclareMathOperator{\frank}{flat-rank}

\newcommand{\x}{\mathbf{x}}

\newcommand{\E}{\mathbb{E}}

\DeclareMathOperator{\codim}{codim}
\newcommand{\randm}{\widetilde{m}}

\newcommand{\poly}{\mathrm{poly}}
\renewcommand{\epsilon}{\varepsilon}

\begin{document}

\title[Which groups are amenable to proving exponent two for matrix multiplication?]
{Which groups are amenable to proving \\ exponent two for matrix
multiplication?}
\author{Jonah Blasiak}
\address{\normalfont J.B.: Department of Mathematics, Drexel University, Philadelphia, PA, USA.
\texttt{jblasiak@gmail.com}. Supported by NSF grant DMS-1600391. All authors
also thank AIM for hosting a SQuaRE, during which this work was developed. }

\author{Thomas Church}
\address{\normalfont T.C.: Department of Mathematics, Stanford University, Palo Alto, CA, USA.
\texttt{church@math.stanford.edu}. Supported by NSF grant DMS-1350138, the
Alfred P.\ Sloan Foundation, and the Frederick E.\ Terman Fellowship.}

\author{Henry Cohn}
\address{\normalfont H.C.: Microsoft Research New England, One Memorial Drive, Cambridge, MA, USA.
\texttt{cohn@microsoft.com}. }

\author{Joshua A. Grochow}
\address{\normalfont J.A.G.: Departments of Computer Science and Mathematics,
University of Colorado, Boulder, CO, USA. \texttt{jgrochow@colorado.edu}.
Supported by a Santa Fe Institute Omidyar Fellowship and NSF grant DMS-1750319.}
\author{Chris Umans}
\address{\normalfont C.U.: Department of Computing and Mathematical Sciences,
California Institute of Technology, Pasadena, CA, USA. \texttt{umans@cs.caltech.edu}.
Supported by NSF grant CCF-1423544 and a Simons Foundation Investigator grant.}

\date{}
\setcounter{page}{0}

\begin{abstract}
The Cohn--Umans group-theoretic approach to matrix multiplication suggests
embedding matrix multiplication into group algebra multiplication, and
bounding $\omega$ in terms of the representation theory of the host group.
This framework is general enough to capture the best known upper bounds on
$\omega$ and is conjectured to be powerful enough to prove $\omega = 2$,
although finding a suitable group and constructing such an embedding has
remained elusive. Recently it was shown, by a generalization of the proof
of the Cap Set Conjecture, that \emph{abelian} groups of {\em bounded
exponent} cannot prove $\omega = 2$ in this framework, which ruled out a
family of potential constructions in the literature.

In this paper we study \emph{nonabelian} groups as potential hosts for an
embedding. We prove two main results:
\begin{enumerate}
\item We show that a large class of nonabelian groups---nilpotent groups
    of bounded exponent satisfying a mild additional condition---cannot
    prove $\omega = 2$ in this framework. We do this by showing that the
    shrinkage rate of powers of the augmentation ideal is similar to the
    shrinkage rate of the number of functions over $(\Z/p\Z)^n$ that are
    degree $d$ polynomials; our proof technique can be seen as a
    generalization of the polynomial method used to resolve the Cap Set
    Conjecture.
\item We show that symmetric groups $S_n$ cannot prove nontrivial bounds
    on $\omega$ when the embedding is via three Young
    subgroups---subgroups of the form $S_{k_1} \times S_{k_2} \times
    \dotsb \times S_{k_\ell}$---which is a natural strategy that includes
    all known constructions in $S_n$.
\end{enumerate}
By developing techniques for negative results in this paper, we hope to
catalyze a fruitful interplay between the search for constructions proving
bounds on $\omega$ and methods for ruling them out.
\end{abstract}

\maketitle

\thispagestyle{empty}
\newpage

\mysection{Introduction}

One of the most prominent open problems in algorithms is to determine the
exponent, $\omega$, of matrix multiplication, the smallest real number such
that $n \times n$ matrices can be multiplied in $O(n^{\omega+\epsilon})$
operations for all $\epsilon > 0$. The exponent $\omega$ controls the
algorithmic complexity of nearly all algorithmic linear algebra problems, and
the best upper bounds on the complexity of numerous other problems seemingly
unrelated to matrix multiplication are expressed in terms of $\omega$. It is
a folklore conjecture that $\omega = 2$, and the quest to prove this has
extended nearly 50 years, sparked by Strassen's 1969 discovery that $\omega <
2.81$. The current best upper bound is $\omega < 2.372864$, due to Le Gall in
2014 \cite{legall}, and building on \cite{CW, Stothers, williams}.

In 2003, Cohn and Umans \cite{CU03} proposed a method for proving upper
bounds on $\omega$ via reduction to group-algebra multiplication. The recent
computer-assisted arguments \cite{Stothers, williams, legall} in the style of
Coppersmith--Winograd \cite{CW} can all be captured by the group-theoretic
approach \cite{CKSU, fuKleinberg}. Indeed, these constructions all can be
viewed as giving families of subsets satisfying the simultaneous triple
product property (STPP) \cite{CKSU} (recalled as Definition~\ref{def:STPP}
below) in \emph{abelian groups of bounded exponent}. This family of groups
arises because the constructions typically work in groups like $(\Z/m\Z)^n$ where
$n \to \infty$ and $m$ can be optimized over, and the optimization results in
$m$ being fixed as $n \to \infty$. Ambainis, Filmus, and Le Gall \cite{AFLG15} showed that the
Coppersmith--Winograd family of constructions could not yield a bound on
$\omega$ better than $2.3078$. The authors, together with Naslund and Sawin,
extended the recent resolution of the Cap Set Conjecture \cite{EG, CLP} to
show a result that is philosophically similar to \cite{AFLG15} (but
technically incomparable): STPP constructions in abelian groups of bounded
exponent cannot show $\omega=2$ \cite{BCCGNSU}.

However, the group-theoretic approach has the advantage of being extremely
general. Part of the allure of this approach is that there is rich array of
constructions to try---especially in nonabelian groups---that make contact
with well-studied topics in mathematics. For example, constructions building
on the ``triangle construction'' of \cite{CU03} turn on the combinatorics of
permutations, while constructions building on the ``Lie pseudo-exponent two''
construction of \cite{CU03} depend on algebraic geometry in finite
characteristic. One can also try to use knowledge from representation theory
to construct families of groups tailored so that their representations are
small, while still supporting matrix multiplication via the Cohn--Umans
embedding. Beyond the elementary fact shown in \cite{CU03} that abelian
groups cannot prove non-trivial bounds on $\omega$ via embedding a {\em
single} matrix-multiplication instance, there has been very little to guide
the search for a construction in this broad approach. In this paper, our
primary purpose is to narrow the search for which nonabelian groups might be
fruitful in this endeavor. As the group-theoretic approach is very natural,
and captures all the known best algorithms for matrix multiplication, we view
such negative results (ruling out families of groups) in the same spirit as
results concerning LPs or SDPs that rule out certain classes of algorithms.

We make the following progress in understanding which groups are amenable to
proving $\omega = 2$ via the group-theoretic approach. We prove two main
results:
\begin{enumerate}
\item We prove upper bounds on STPP constructions in finite nilpotent
    groups $G$ of bounded exponent (and satisfying a mild additional
    condition; see Section \ref{sec:pgroups}) of the form $|G|^{1 -
    \epsilon}$ for a constant $\epsilon > 0$ (Theorem~\ref{thm:nilpotent}).
    This rules out proving $\omega = 2$ in such groups. The ``bounded
    exponent'' condition is essentially inherited from the abelian case,
    where one cannot rule out large cyclic groups or powers thereof by
    our methods (see Section~\ref{sec:techniques} for more details of the
    methods).

  In the taxonomy of finite groups, nilpotent groups are ``just short of''
  solvable groups.  Within the scope of our technique, solvable groups are
  perhaps the most fundamental and natural class one might reasonably hope
  to rule out.  Recall that solvable groups are built from abelian groups
  as follows: all abelian groups are solvable, and if $N$ is a normal
  subgroup of $G$ and $N$ and $G/N$ are solvable then so is $G$. We
  highlight extending these upper bounds to solvable groups as an important
  open problem.

\item We prove that no three Young subgroups can prove $\omega = 2$ in the
    symmetric or alternating groups. In the taxonomy of finite groups,
    these non-abelian simple or near-simple groups are in some sense ``at
    the other end of the spectrum'' from abelian, nilpotent, or solvable
    groups. While our result does not fully rule out proving $\omega = 2$
    in such groups in the way that upper bounds on the size of arbitrary
    STPP constructions would, it does rule out proving it via perhaps the
    most natural type of construction.
\end{enumerate}

Altogether, these results together with \cite{BCCGNSU} begin to narrow the
choices for proving $\omega = 2$ in intriguing and useful ways. While it is
certainly not a foregone conclusion that the group-theoretic approach is
capable of proving $\omega = 2$ (or even that $\omega = 2$ in the first
place!), there seems to now be a useful interplay between positive and
negative results that constitutes a mathematically interesting research
program with a chance for a significant payoff.

\subsection{Techniques} \label{sec:techniques}
Our main technique is the \emph{slice rank} method described in
\cite{BCCGNSU}. However, that method is an application of the polynomial
method, and it is initially not clear how to extend it to the nonabelian
setting. Instead of considering polynomials graded by degree (as in the
abelian setting), we show that the powers of the augmentation ideal of a
nilpotent group are a suitable nonabelian replacement for the grading by
degree. If $I$ is the augmentation ideal, then the replacements for
``polynomials of degree $k$'' is the space $I^k / I^{k+1}$. We identify the
\emph{shrinkage rate} of these spaces (as $k$ increases) as a key quantity
for bounding the slice rank, and hence the size of STPP constructions, in
such groups. We show a concentration inequality for these dimensions strong
enough to give our main theorem on nilpotent groups,
Theorem~\ref{thm:nilpotent}.

It is possible that this behavior occurs in powers of the augmentation ideal
in other groups as well; this suggests a concrete strategy for proving strong
upper bounds for groups beyond the ones we have considered.

Our result for symmetric groups is a fairly delicate induction. It is
intriguing that altering the setup in certain small ways---for example by
considering the direct product of {\em two} symmetric groups---breaks the
argument. Do any of these alterations suggest ways to obtain {\em positive
results}? We discuss these questions in Section \ref{sec:symmetric}.

\mysubsection{Related work} The techniques and results of this paper have
significant overlap with those of Petrov \cite{petrov}. He also uses powers
of the augmentation ideal and proves a result about products of subspaces in
a group ring (similar to our Proposition \ref{prop:codim}) to obtain his
upper bounds. Indeed, in retrospect, our proof specializes in the case of the
unitriangular group (Example~\ref{ex:UT}) to precisely that given by Petrov
in \cite[Section~7]{petrov}. However, by realizing the argument in terms of
powers of the augmentation ideal and relating this to the $p$-lower central
series (see Proposition~\ref{prop:augVSplcs}), we are able to obtain general
results for all $p$-groups. Moreover, by putting the argument
in the context of slice rank, we are able to show that our bounds in general
are tight (see Appendix~\ref{sec:tight}).

As a result of this more general approach, we identify two natural structural
properties of $p$-groups that allow us to rule out showing $\omega=2$ in
groups satisfying these properties. One of these is bounded nilpotency
class---a standard group-theoretic notion---but the other, about the growth
rate of ``$p$-degrees'' (see Definition~\ref{def:pdegrees}), appears to be
new and may be interesting in its own right.

We also show how to extend slice rank upper bounds from a normal subgroup to
its parent group (Lemma~\ref{lem:normal}), which is a very general tool. In
this paper, we use this tool to extend from $p$-groups to general nilpotent
groups.

Sawin \cite{Sawin} also gives a general result. He shows that for any
nontrivial group  $G$, the size of a multiplicative matching in  $G^n$ is at
most $\delta^n |G|^n$, where  $\delta < 1$ is a constant that depends on $G$
but not  $n$. However, this bound is never of the form $|G|^{1 - c}$, which
is what is needed to rule out proving $\omega=2$ in a family of groups
(unless $|G|=O(1)$ and the family is $\{G^n\}$). In contrast, our results
rule out proving $\omega=2$ in many natural families of groups, including
ones with known non-trivial constructions.

Even apart from the connection with matrix multiplication, the question of
extremal multiplicative matchings and related objects in groups is
interesting in its own right, and has been the subject of a number of recent
works \cite{CLP, EG, Sawin, naslundSawin, KSS, Norin, Pebody, foxLovasz,
greenSarkozy, ellenbergSumset, geShangguan, kimOun, ASU, BCCGNSU, A,
dvirEdelman}.

\mysection{Preliminaries}

\mysubsection{Multiplicative matchings}

The following definition coincides with what were called {\em tricolored
sum-free sets} in \cite{BCCGNSU}. While the ``sum-free'' terminology works
well in abelian groups, we find the ``matching'' terminology adopted by
Aaronson \cite{A} and Sawin \cite{Sawin} clearer, especially when the
underlying group is non-abelian as it frequently is in this paper.

\begin{definition}[{Multiplicative matching \cite[Def.~3.1]{BCCGNSU}}]
A {\em multiplicative matching} in a group $G$ is given by three sequences $(s_1, \dotsc, s_n), (t_1, \dotsc, t_n), (u_1, \dotsc, u_n)$ of
elements of $G$
such that
\[
s_it_ju_k = 1 \Longleftrightarrow i = j =k.
\]
The {\em cardinality} of this multiplicative matching is $n$.
\end{definition}

\mysubsection{The group-theoretic approach}

The group-theoretic approach to bounding the exponent of matrix
multiplication amounts to reducing matrix multiplication to multiplication in
the {\em group algebra} $\C[G]$ for finite groups $G$. The reduction is
carried out via three subsets of $G$ that satisfy the {\em triple product
property}:

\begin{definition}[Triple Product Property (TPP)]
Three subsets $S, T, U$ of a finite group $G$ satisfy the {\em triple product
  property} if
\[stu = 1 \Longleftrightarrow s = t = u = 1\]
for all $s \in Q(S), t \in Q(T), u \in Q(U)$. Here $Q(S) = \{xy^{-1}:
x, y \in S\}$ is the {\em quotient set} of $S$.
\end{definition}

Given $S, T, U \subseteq G$ that satisfy the triple product property, one can
reduce $|S| \times |T|$ by $|T| \times |U|$ matrix multiplication to
$\C[G]$-multiplication. A key fact is that $\C[G] \cong M_{d_1}(\C) \oplus
\dotsb \oplus M_{d_k}(\C)$, where the $d_i$ are the dimensions of the
irreducible representations of $G$ (hence $\sum_i d_i^2 = |G|$). The
reduction thus gives a scheme for multiplying matrices by performing a number
of (hopefully) smaller matrix multiplications, yielding a recurrence that
proves upper bounds on $\omega$:

\begin{theorem}[\cite{CU03}] \label{thm:TPP}
If $S, T, U \subseteq G$ satisfy the triple product property, then
\begin{equation} \label{eq:TPP}
(|S| |T| |U|)^{\omega/3} \leq \sum_i d_i^{\omega},
\end{equation}
where the $d_i$ are the dimensions of the irreducible representations
of $G$.
\end{theorem}

One hopes for $S, T, U$ to be large subsets, and the $d_i$ to be small, so
that this equation forces $\omega$ to be small. In this paper we identify a
useful {\em necessary} condition for a Triple Product Property construction
to prove nontrivial bounds on $\omega$ (i.e., $\omega < 3$), which we use in
Section~\ref{sec:symmetric}:

\begin{proposition}
\label{prop:nec-TPP} If $S, T, U \subseteq G$ satisfy the \ifconfversion
\else Triple Product Property \fi and \ifconfversion
\else
\[\frac{|G|}{(|S||T||U|)^{2/3}} \ge \mbox{\# conjugacy
    classes of $G$},\]
    \fi 
then \eqref{eq:TPP} is satisfied by all $\omega > 0$ (and thus cannot even prove $\omega < 3$).
\end{proposition}

\ifconfversion
\else
\begin{proof}
A well-known fact is that the number of inequivalent irreducible
representations of $G$ is equal to the number $k$ of conjugacy classes of
$G$. Since $\sum_i d_i^2 =|G|$ we find that $d^2_{\max} \ge |G|/k$ (indeed
the {\em average} of $d_i^2$ is at least this large). By assumption, we have
\[(|S||T||U|)^{2/3} \le |G|/k \le d_{\max}^2\]
Exponentiating both sides by $\omega/2$ gives $(|S||T||U|)^{\omega/3} \le
d_{\max}^\omega$. Since $d_{\max}^\omega\le \sum_i d_i^\omega$ we find
that \eqref{eq:TPP} is satisfied by any positive $\omega$, as
claimed.
\end{proof}
\fi 

It {\em is} possible to prove good upper bounds on $\omega$ in this
framework, and this was first done in \cite{CKSU}, via wreath product groups
$G^n \rtimes S_n$. It turns out that the apportionment to subsets $S, T, U$
in all of these constructions can be described via several ``simultaneous''
triple product property constructions within $G$.

\begin{definition}[Simultaneous Triple Product Property (STPP)] \label{def:STPP}
Triples of subset $S_i, T_i, U_i$ of $G$ satisfy the {\em simultaneous triple
product property} if
\begin{enumerate}
\item for each $i$, the triple $S_i, T_i, U_i$ satisfies the triple
  product property in $G$, and
\item for all $i, j, k$ and $s \in S_i, s' \in S_j,  t \in T_j, t' \in
  T_k, u \in U_k, u' \in U_i$ we have
\[s^{-1}s't^{-1}t'u^{-1}u' = 1 \Rightarrow i = j = k.\]
\end{enumerate}
\end{definition}

One can understand an STPP construction as a means of reducing several
independent matrix multiplications (of format $|S_i| \times |T_i|$ by $|T_i|
\times |U_i|$) to a single $\C[G]$ multiplication. Via either the wreath
product machinery of \cite{CKSU} or the Asymptotic Sum Inequality
\cite{Schonhage}, we obtain the following theorem.

\begin{theorem}[\cite{CKSU}]
If $S_i, T_i, U_i \subseteq G$ satisfy the simultaneous triple product property, then
\begin{equation} \label{eq:STPP}
\sum_i (|S_i| |T_i| |U_i|)^{\omega/3} \leq \sum_i d_i^\omega,
\end{equation}
where the $d_i$ are the dimensions of the irreducible representations
of $G$.
\end{theorem}

STPP constructions generalize TPP constructions and indeed are the
most general kind of construction in the group-theoretic framework. It
is STPP constructions that can be made to mimic the
Coppersmith--Winograd result \cite{CW} and the recent improvements \cite{Stothers, williams, legall}, thus
capturing the best known bounds on $\omega$. An important constraint on
STPP constructions, which can easily be derived from the definition, are the {\em packing bounds} which assert that
\ifconfversion
\else
\[\sum_i |S_i||T_i| \le |G|, \;\;\;\; \sum_i |T_i||U_i|
  \le |G|,\;\;\;\; \text{and } \sum_i |S_i||U_i| \le |G|.\]
  \fi 
One can hope to obtain constructions that come very close to this bound:

\begin{definition}[{Packing bound \cite[Def.~2.3]{BCCGNSU}}]
A family of STPP constructions in groups $G$ with $|G| \rightarrow
\infty$ \emph{meets the packing bound} if
\[\sum_i |S_i||T_i| \ge |G|^{1 - o(1)}, \;\;\;\; \sum_i |T_i||U_i|
  \ge |G|^{1 - o(1)},\;\;\;\; \text{and }\, \sum_i |S_i||U_i| \ge |G|^{1 - o(1)}.\]
\end{definition}

This provides a useful {\em necessary} condition for a STPP construction to
prove $\omega = 2$:

\begin{theorem}[{\cite[Lemma~2.4]{BCCGNSU}}]
Any family of STPP constructions that does not meet the packing bound cannot
imply $\omega = 2$ via Inequality \ref{eq:STPP}.
\end{theorem}

Finally we can state the connection between the group theoretic approach and
multiplicative matchings, which was proved in \cite{BCCGNSU}:

\begin{theorem}
Any family of STPP constructions in groups $G$ with $|G| \rightarrow \infty$
that meets the packing bound implies the existence of multiplicative matchings in $G^N$ with
cardinality $|G|^{N(1 - \epsilon)}$ for {\em arbitrarily small} $\epsilon
>0$, by choosing sufficiently large $|G|$ and $N$.
\end{theorem}

Thus for a family of groups $G$ to host an embedding of matrix multiplication
that proves $\omega = 2$ in the group-theoretic framework, powers of $G$ must
contain multiplicative matchings that are nearly the largest possible.

\mysubsection{Slice rank of tensors} Given three finite sets $X,Y,Z$ and a
field $\F$ we think of functions $F\colon X \times Y \times Z \to \F$ as
3-tensors. The \emph{slice rank} of such an $F$ (introduced by Tao
\cite{taoBlog}; see also \cite{BCCGNSU, taoSawin}) is the smallest $r$ such
that there are functions $f_i, g_i$ for which we can write
\[
F(x,y,z) = \sum_{i=1}^a f_i(x,y) g_i(z) + \sum_{i=a+1}^b f_i(x,z) g_i(y) + \sum_{i=b+1}^r f_i(y,z) g_i(x).
\]
\ifconfversion \else When $F$ only takes the values $\{0,1\}$, we may
consider its slice rank over various fields $\F$, in which case we write
$\slicerank_\F(F)$, though usually $\F$ will be clear from context. We
sometimes refer to {\em characteristic $p$ slice rank} to emphasize that the
characteristic is playing a critical role in the bound under discussion.
\fi 

Given an algebra $\mathcal{D}$ (such as a group ring $\F[G]$), its
multiplication tensor $\mult{\mathcal{D}}$ relative to a basis $x_1,
\dotsc, x_{\dim \mathcal{D}}$ is defined by \ifconfversion
\else
\[
x_i x_j = \sum_k \mult{\mathcal{D}}(i,j,k) x_k.
\]
\fi In particular, for a group $G$, if we choose the group elements as the
basis of $G$, we find that $\mult{\F[G]}$ only has zero-one values, and so can in
fact be defined over any field; when we wish to leave the field unspecified
we thus write $\mult{G}$ for the multiplication tensor of a group $G$.

If we think of a tensor $F \colon X \times Y \times Z \to \F$ as an element
of $\F^X \otimes \F^Y \otimes \F^Z$, then the slice rank is invariant
under change of basis in each of the three factors (that is, the action of
the group $\GL_{|X|}(\F) \times \GL_{|Y|}(\F) \times \GL_{|Z|}(\F)$). Thus,
even if we have a function $F \colon X \times X \times X \to \F$, we may
choose different bases for each of the three copies of $\F^X$ and reason
about the slice rank of $\F$ in our favorite three bases, which will be a
useful trick.

Slice rank gives us a means to bound the cardinality of multiplicative
matchings from above:

\begin{proposition}[Tao \cite{taoBlog}]
\label{prop:tao} If $G$ contains a multiplicative matching of cardinality $m$
then $\slicerank(\mult{G}) \geq m$ (over any field).
\end{proposition}

In summary, we have the following implications: $\omega = 2$ via STPP in family $G$ $\Longrightarrow$
nearly-largest-possible multiplicative matchings in powers of $G$
$\Longrightarrow$ slice rank of $G$-multiplication tensor is at least $|G|^{1
- o(1)}$, which is encapsulated by the following corollary:

\begin{corollary}[Key corollary] \label{cor:key}
Given a family of groups $G$ with $\slicerank(\mult{G}) \leq
|G|^{1-\Omega(1)}$, no STPP construction in this family can prove $\omega=2$
via Inequality (\ref{eq:STPP}).
\end{corollary}

In the next section we use this to rule out proving $\omega = 2$ in a large
class of nilpotent groups.

\mysection{Ruling out a large class of nilpotent groups}

We begin with a general lemma about slice rank and a consequence about the
slice rank of algebras over a field, which may have further uses. Next we
prove a result ruling out proving $\omega = 2$ in a large class of
$p$-groups, and then use the fact that nilpotent groups are direct products
of $p$-groups, together with machinery for passing to group extensions, to
rule out proving $\omega = 2$ via a large class of nilpotent groups.

\mysubsection{Slice rank of algebras}

In the next lemma and proposition we give an appealing sufficient condition
for establishing upper bounds on the slice rank of algebras over a field.

\begin{lemma}
\label{lem:slice rank restrict}
For a function $F\colon X \times Y \times Z \to \F$ and a subset $\hat{X}
\subseteq X$, we have
\ifconfversion
\else
\[\slicerank(F) \le \slicerank(F|_{\hat{X} \times Y \times Z})+ |X|-|\hat{X}|.\]
\fi
A similar statement holds for restricting  $Y$ or  $Z$.
\end{lemma}

\ifconfversion
\else
\begin{proof}
Let  $\hat{F} = F|_{\hat{X} \times Y \times Z}$ and $\hat{r} =
\slicerank(\hat{F})$.
There exist functions  $\hat{f}_i, f_i, \hat{g}_i, g_i$ such that
\[
\hat{F}(x,y,z) = \sum_{i=1}^a \hat{f}_i(x,y) g_i(z) + \sum_{i=a+1}^b
\hat{f}_i(x,z) g_i(y) + \sum_{i=b+1}^{\hat{r}} f_i(y,z) \hat{g}_i(x)
\quad \text{ for $x, y, z \in \hat{X} \times Y \times Z$}.
\]
For $i \le a$, let $f_i$ denote the extension of  $\hat{f}_i$ from the
domain $\hat{X} \times Y$ to  $X \times Y$ which is zero whenever  $x
\in X
\setminus \hat{X}$;
define similar extensions $f_i(x,z)$ of $\hat{f}_i(x,z)$ for $a < i
\le b$ and $g_i(x)$ of $\hat{g}_i(x)$ for  $b < i \le \hat{r}$.
For  $x \in X \setminus \hat{X}$, let $h_x(y,z) = F(x, y,z)$.
We then have the following expression for $F$, which proves the lemma:
\[
F(x,y,z) = \sum_{i=1}^a f_i(x,y) g_i(z) + \sum_{i=a+1}^b f_i(x,z)
g_i(y) + \sum_{i=b+1}^{\hat{r}} f_i(y,z) g_i(x)  +  \sum_{w \in X
\setminus \hat{X}} h_w(y,z) \delta_{w, x}. \qedhere
\]
\end{proof}
\fi 

\begin{proposition}
\label{prop:codim} Let $\D$ be a finite-dimensional algebra over a field. If
$A,B,C$ are subspaces of $\D$ satisfying $A\cdot B\subseteq C$, then
$\slicerank(\mult{\mathcal{D}}) \leq \codim A+\codim B+\dim C$.
\end{proposition}

\begin{proof}
Let $d = \dim D$, $d_A = \dim A$, $d_B = \dim B$, and $d_C = \dim C$.
Let $x_1, \dotsc, x_d$ be a basis for $\D$ such that the prefix $x_1,
\dotsc, x_{d_A}$ is a basis for $A$, let $y_1, \dotsc, y_d$ be a basis
for $\D$ such that $\Span\{y_1, \dotsc, y_{d_B}\} = B$, and let $z_1,
\dotsc, z_d$ be a basis of $\D$ such that $\Span\{z_1, \dotsc, z_{d_C}\}
= C$. In these bases, $T=\mult{\mathcal{D}}$ looks like
\ifconfversion
\else
\[
x_i y_j = \sum_{k=1}^d T(i,j,k) z_k,
\]
\fi
and $\hat{T} := T|_{[d_A] \times [d_B] \times \{d_C + 1, \dots, d\}} = 0$.
Hence by Lemma \ref{lem:slice rank restrict},
\[\slicerank(T) \le d-d_A+d-d_B + d_C + \slicerank(\hat{T}) = \codim A
+ \codim B + \dim C. \qedhere\]
\end{proof}

\subsubsection{The structure of $p$-groups, and the augmentation ideal} \label{sec:pgroups}

To bound the slice rank of a $p$-group, we will use the following special
case of Proposition~\ref{prop:codim}:

\begin{lemma}
\label{lem:ideal} Let $G$ be a finite group, $\F$ a field, and $I$ an ideal
in $\F [G]$. Then for any $a,b \in \N$,
\[
\slicerank(\mult{G}) \leq \codim I^a + \codim I^b + \dim I^{a+b}.
\]
\end{lemma}

We'll apply this lemma to the augmentation ideal.

\begin{definition}[Augmentation ideal]
The \emph{augmentation ideal} $I\subseteq \F[G]$ is the kernel of the
natural map $\F[G]\onto \F$ defined on the basis of group elements by
$g \mapsto 1$ for all $g \in G$.
\end{definition}

The augmentation ideal is linearly spanned by the group algebra elements of
the form $g-1$ for $g \in G$. It is a standard fact that if $G$ is a
$p$-group then the augmentation ideal of $\F_p[G]$ is nilpotent, meaning that
$I^{e+1} = \Span\{ x_1 x_2 \dotsc x_{e+1} : x_i \in I \} = 0$ for some $e$
(see, e.\,g., \cite{jennings}). Taking the minimal such $e$, we have nested
subspaces $\F[G]\supsetneq I \supsetneq I^2 \supsetneq \dotsb \supsetneq
I^{e} \supsetneq I^{e+1}=0$. To get the most mileage out of
Lemma~\ref{lem:ideal}, we wish to choose $a,b$ so that  $I^a$ and  $I^b$ are
large while still keeping $I^{a+b}$ small. If $P$ is the distribution on
$\{0,\dotsc,e\}$ defined by taking $P(i)$ proportional to $\dim I^i /
I^{i+1}$, we want to find $a,b$ such that $\sum_{i<a} P(i)$, $\sum_{i < b}
P(i)$, and $\sum_{a+b \leq i} P(i)$ are all small. We therefore seek some
sort of concentration inequality in which $P(i)$ is concentrated around
values of $i$ near the middle. We will use the $p$-degrees of a $p$-group,
defined below, to prove just such a concentration inequality in
Theorem~\ref{thm:concentration}.

With each $p$-group $G$ satisfying $\abs{G}=p^n$, we associate a sequence
$(r_1,\ldots,r_\ell)$ of nonnegative integers, which we call the
\emph{$p$-degrees} of $G$, with $r_1+\cdots+r_\ell=n$. These $p$-degrees are
defined in terms of the $p$-lower central series or Jennings series of $G$;
this is a variant of its lower central series, which has the advantage of
controlling the powers of the augmentation ideal in $\F_p[G]$
(see \cite{jennings}).

\begin{definition}[{$p$-degrees, $p$-lower central series (see, e.\,g., \cite{jennings})\footnote{N.B.:
There is another standard variant of the lower central series with a very
similar name, the ``lower central exponent-$p$ series,'' defined by $H_i =
[G, H_i] H_{i-1}^{(p)}$. The difference in grading between the $p$-lower
central series we consider and the $H_i$ is crucial.}}] \label{def:pdegrees}
\ifconfversion
\else The \emph{$p$-lower central
series} $\Gamma_i=\Gamma_i(G)$ is defined to be the smallest filtration
$G=\Gamma_1\supseteq \Gamma_2\supseteq \cdots$ such that
\[
[\Gamma_i,\Gamma_j]\subseteq \Gamma_{i+j} \qquad \text{ and } \qquad g\in \Gamma_i\implies g^p\in \Gamma_{ip}.
\]
Equivalently \cite[Theorem~5.5]{jennings}, we may define the $\Gamma_i$
inductively by $\Gamma_1 = G$ and $\Gamma_i = [G, \Gamma_{i-1}]
\Gamma_{\lceil i/p \rceil}^{(p)}$, where for any subgroup $H$, $H^{(p)}$
denotes the subgroup generated by $\{h^p : h \in H\}$.
\fi 
The quotient $\R_j=\Gamma_j/\Gamma_{j+1}$ is an $\F_p$-vector space (since
$\Gamma_{j+1} \supseteq \Gamma_{jp}$). The \emph{$p$-degrees} of $G$ are the
sequence of dimensions $r_j\coloneq \dim \R_j$.

When $G$ is a $p$-group, the $p$-lower central series terminates,
meaning that $\Gamma_{\ell+1} = 1$; the minimal such $\ell$ is called
the \emph{length} of the $p$-lower central series of $G$, and
$r_1+\cdots+r_\ell = n$ if $|G|=p^n$.
\end{definition}

Our results apply whenever the $p$-degrees do not decrease ``too rapidly'' or
when they have bounded variance; these two conditions are formalized as
follows.
Given a real non-negative vector $r=(r_1,\ldots,r_\ell)$ with $\sum_i r_i =
n$, let $\rho_i = r_i/n$, and let $X_r$ be the random variable that takes value
$i\in \{1,\ldots,\ell\}$ with probability $\rho_i$.

\begin{definition}[Linear expectation]
We say that a family of such vectors $r$ has \emph{linear expectation} if there exists
some universal constant $c>0$ such that $\E(X_r) \geq \ell/c$.
\end{definition}

\begin{definition}[Bounded variance]
We say that a family of such vectors $r$ has \emph{bounded
variance} if there exists some universal constant $M$ such that $\Var(X_r) \leq M$.
\end{definition}

\ifconfversion
\else
We can now state our main theorem for $p$-groups:
\fi 

\begin{theorem}[Main theorem for $p$-groups] \label{thm:pgroups}
STPP constructions in families of $p$-groups of bounded exponent cannot
achieve $\omega=2$ if they have either
\begin{enumerate}
\item \label{case:bounded_var} $p$-degrees of bounded variance, or
\item \label{case:linear-expectation} $p$-degrees of linear expectation.
\end{enumerate}
\end{theorem}
The proof is given after the next section, which shows how the $p$-degrees
control the dimension of the powers of the augmentation ideal.

As a corollary, we get an even simpler condition, which may be useful in
further applications:

\begin{corollary} \label{cor:bounded_length}
STPP constructions in families of $p$-groups of bounded exponent and with
bounded length of $p$-lower central series cannot achieve $\omega=2$.
\end{corollary}

\begin{proof}
If the length of the $p$-lower central series is bounded, then \emph{a
fortiori} its variance is bounded, so this follows from
Theorem~\ref{thm:pgroups}(\ref{case:bounded_var}).
\end{proof}

\mysubsection{$p$-degrees and a polynomial method for $p$-groups} Given a
$p$-group $G$ with $p$-degrees $(r_1,\ldots,r_\ell)$, consider $n$ variables
labeled $x_{j,i}$ where $1\leq j\leq \ell$ and $1\leq i\leq r_j$.
\ifconfversion
\else Let $X$ denote the set of monomials $\x^\m=\prod
x_{j,i}^{m_{j,i}}$ with the property that $0 \leq m_{j,i}<p$ for all $j,i$.
\fi 
We define a \emph{weighted degree} on such monomials by $\deg x_{j,i}=j$ and
thus $\deg \x^\m=\sum_{j,i} jm_{j,i}$.

\begin{proposition}[{\cite[Theorem~3.7]{jennings}}]
\label{prop:augVSplcs} Given a finite $p$-group $G$, let $I$ be the
augmentation ideal of $\F_p[G]$. With notation as above, \ifconfversion
\else
 the dimension of $I^k$ is the number of monomials $\x^{\m}\in X$ with $\deg(\x^{\m})\geq k$.
\fi 
\end{proposition}

\begin{proof}
For each $j$, choose $g_{j,i}\in \Gamma_j\subseteq G$ descending to a basis for
$\R_j$. Set $x_{j,i}=g_{j,i}-1\in \F_p[G]$, so every monomial $\x^\m$ defines
an element of $\F_p[G]$. Then a basis for $I^k$ is given by those $\x^\m\in
X$ with $\deg \x^\m\geq k$ \cite[Theorem 3.2]{jennings} (see \cite{Quillen}
for a conceptual explanation). In particular, the degree of a product of such
basis elements is the sum of the degrees of the basis elements, as would be
expected for ordinary degree.
\end{proof}

We'll use Hoeffding's Inequality and Proposition~\ref{prop:augVSplcs} to
estimate the dimension of $I^k$:

\begin{theorem}
\label{thm:concentration} Given a $p$-group $G$ of order $p^n$ with
$p$-degrees $(r_1,\ldots,r_\ell)$, define
\begin{equation}
\label{eq:delta-G}
\delta_G=\frac{(\sum_j j r_j)^2}{\sum_j j^2 r_j}.
\end{equation}
Then $\slicerank(\mult{G})\leq 3 \abs{G}e^{-\delta_G/18}$. In particular,
$\slicerank(\mult{G}) \leq p^n/e^{\Omega(\delta_G)}$.
\end{theorem}

\begin{proof}
Let $s = (p-1)\cdot\sum_j jr_j$ be the maximum degree of any $\x^\m\in X$.
Applying Lemma~\ref{lem:ideal} with $a=b=s/3$ gives the bound
$\slicerank(\mult{G}) \leq 2\codim I^{s/3}+\dim I^{2s/3}$. The distribution
of degrees in Proposition~\ref{prop:augVSplcs} is symmetric about $s/2$, so
we have $\codim I^{\frac{s}{2}-t}=\dim I^{\frac{s}{2}+t}$. In particular,
this gives $\slicerank(\mult{G}) \leq 3\dim I^{2s/3}$. It remains to bound
$\dim I^{2s/3}$ using Proposition~\ref{prop:augVSplcs}.

To generate a random monomial $\x^\m\in X$, we may independently choose
$m_{j,i}\in \{0,\ldots,p-1\}$ and form the product $\x^\m=\prod
x_{j,i}^{m_{j,i}}$. Since $\deg \x^\m=\sum (\deg x_{j,i})m_{j,i}=\sum j
m_{j,i}$, we can express $\deg \x^\m$ as the sum of independent random
variables $\randm_{j,i}$, where $\randm_{j,i}=\deg x_{j,i}^{m_{j,i}}$ is a
uniform random variable taking values in $\{0,j, 2j, \ldots, {(p-1)j}\}$.
Each $\randm_{j,i}$ has a bounded range of $j(p-1)$ and their sum has
expectation $s/2$. So by Hoeffding's Inequality, we obtain for all $t > 0$ that
\[\dim I^{\frac{s}{2}+t}=\big|\big\{\x^\m \in X: \deg \x^\m \ge \frac{s}{2} +t \big\}\big|\leq
  \frac{\abs{G}}{e^{2t^2/(\sum_j j^2(p-1)^2r_j)}}\]
using that $\abs{X}=\abs{G}$. Taking $t=s/6$, we find $\dim I^{2s/3}\leq
\abs{G} / e^{\delta_G / 18}$, proving the proposition.
\end{proof}
We now apply these bounds on four important examples. In all of these
examples, we think of $p$ as fixed and $n$ as growing.

\begin{example}[Vector spaces]
For $G=(\Z/p\Z)^n$ the only nonzero $p$-degree is $r_1 = n$. Therefore
$\delta_G=r_1^2 / r_1 = r_1=n$, and the bound becomes $p^n/e^{\Omega(n)}$.
This agrees with the bounds of \cite{EG, BCCGNSU} (up to the constants, which
we have not tried to optimize). But regardless of the constant, this is
enough to rule out getting $\omega=2$ via STPP constructions in these groups.
\end{example}

\begin{example} \label{ex:basichomocyclic}
For $G=\Z/{p^n}\Z$ the nonzero $p$-degrees are $r_1 = r_p = r_{p^2} = \dotsb =
r_{p^{n-1}} =1$, so $\delta_G=(\sum_{k=0}^{n-1} p^k)^2 / (\sum_{k=0}^{n-1}
p^{2k})=\Omega(1)$. Therefore the bound we obtain is only
$p^n/e^{\Omega(1)}$. (Another way to see this is that the distribution of
degrees of $\x^\m\in X$ in this case is \emph{uniform} on
$\{0,1,\ldots,p^n-1\}$, so no concentration of measure occurs.) Note that
this bound is nevertheless \emph{optimal} because $\Z/{p^n}\Z$ has a {\em
border} multiplicative matching of size at least $p^n/2$ (see Section
\ref{sec:tight}), and indeed has slice rank $p^n$ (see
Section~\ref{sec:slice_cyclic}).
\end{example}

\begin{example} \label{ex:homocyclic}
For $G=(\Z/{p^k}\Z)^m$ the nonzero $p$-degrees are $r_1 = r_p = r_{p^2} = \dotsb
= r_{p^{k-1}} = m$, so $\delta_G=\Omega(m)$. The resulting bound of
$p^{km}/e^{\Omega(m)}$ agrees with the bound proved in \cite{BCCGNSU} up to
the constant in the $\Omega(\cdot)$. Just as in the previous example, the
factor $e^{\Omega(m)}$ here is sharp. Note that, as for $\Z/{p^k}\Z$ itself,
this does \emph{not} rule out proving $\omega=2$ via an STPP construction in
these groups, as long as $k$ is growing.
\end{example}

\begin{example}[Upper unitriangular matrices] \label{ex:UT}
Let $G$ be the group of $m \times m$ upper unitriangular matrices over
$\F_p$. Then $\abs{G}= p^{(m^2-m)/2}$ and the nonzero $p$-degrees are
$r_1=m-1, r_2={m-2},\ldots,r_{m-1}=1$. Therefore
  \ifconfversion
\else
  \[\delta_G=\frac{\big(\sum_{j=1}^m j(m-j)\big)^2}{\sum_{j=1}^m
      j^2(m-j)} = \Omega(m^2).\]
      \fi 
We obtain a bound on slice rank of the form
\ifconfversion
\else
\[\slicerank(\mult{G})\leq \frac{p^{(m^2-m)/2}}{e^{\Omega(m^2)}}, \]
\fi 
which indeed rules out obtaining $\omega = 2$ via STPP constructions in these
groups.
\end{example}

\mysubsection{Proof of the main $p$-group theorem
(Theorem~\ref{thm:pgroups})} Theorem~\ref{thm:concentration} gives
\[\slicerank(\mult{G}) \leq p^n/e^{\Omega(\delta_G)}\] for any $p$-group $G$.
All that remains is to show that this bound is of the form $|G|^{1 -
\Omega(1)}$ under the hypotheses of Theorem~\ref{thm:pgroups}. Since the
exponent is bounded, $p$ is bounded, and the latter form is equivalent to
saying that $\delta_G \geq \Omega(n)$. Lemmas~\ref{lem:bounded_var} and
\ref{lem:linear-expectation} will cover the two hypotheses of the theorem,
respectively, thus proving the theorem.

\begin{lemma}[Bounded variance] \label{lem:bounded_var}
Given $r=(r_1,\ldots,r_\ell)$ with $n=\sum r_i$ and bounded variance,
\ifconfversion\else
the expression $\delta_G=\frac{(\sum_i i r_i)^2}{\sum_i i^2 r_i}$ from
$\eqref{eq:delta-G}$ satisfies $\delta_G \geq \Omega(n)$.
\fi 
\end{lemma}
\begin{proof}
Let $\rho_i = r_i/n$ be the probability distribution associated to $r$, and
let $X=X_r$ be the random variable that takes value $i$ with probability
$\rho_i$. Then
\[
\Var(X) = \E[X^2] - \E[X]^2 = \big(\sum_i i^2 \rho_i\big) - \big(\sum_i i \rho_i\big)^2.
\]

For the remainder of the proof we find it useful to introduce a
``scale-free'' version of $\delta_G$, namely $\delta'_G = \delta_G / n$. The
conclusion of the lemma is equivalent to $\delta'_G \geq \Omega(1)$, since
this holds if and only if $\delta_G \geq \Omega(n)$. To see that $\delta'_G$
is scale-free, we may use the fact that $\sum_i r_i = n$ to rewrite it as
\begin{equation} \label{eq:delta-prime-G}
\delta'_G = \frac{(\sum_i i r_i)^2}{(\sum_i i^2 r_i)(\sum_i r_i)}.
\end{equation}
This expression makes it clear that if we rescale all of the $r_i$ by some
factor $\alpha$, the quantity $\delta'_G$ is unchanged.

Now, since $\delta'_G$ is scale-free, the quantity remains the same if we
replace each $r_i$ with $\rho_i = r_i/n$. Thus we have
\[
\delta'_G = \frac{(\sum_i i r_i)^2}{(\sum_i i^2 r_i)(\sum_i r_i)}
= \frac{(\sum_i i \rho_i)^2}{(\sum_i i^2 \rho_i)(\sum_i \rho_i)} =
\frac{(\sum_i i \rho_i)^2}{(\sum_i i^2 \rho_i)} =
\frac{(\sum_i i \rho_i)^2}{\Var(X) + \left(\sum_i i \rho_i \right)^2}.
\]
As we have assumed bounded variance, there is some universal constant $M \geq
0$ such that $\Var(X) \leq M$, and we are left with
\[
\delta'_G \geq \frac{(\sum_i i \rho_i)^2}{M + \left(\sum_i i \rho_i \right)^2}.
\]
Let $a = \left(\sum_i i \rho_i\right)^2$, so our bound is $\delta'_G \geq
a/(a + M)$. This function is a non-decreasing function of $a$: its derivative
is $M/(a+M)^2$, which is non-negative since $M$ is. Thus a lower bound on $a$
yields a lower bound on $\delta'_G$; as $a = \left(\sum_i i \rho_i\right)^2
\geq (\sum_i \rho_i)^2 = 1$, we get $\delta'_G \geq 1/(1+M) \geq \Omega(1)$,
as desired.
\end{proof}

\begin{lemma}[Linear expectation] \label{lem:linear-expectation}
Given $r=(r_1,\ldots,r_\ell)$ with $n=\sum r_i$ and linear expectation,
the expression $\delta_G = \frac{(\sum_i i r_i)^2}{\sum_i i^2 r_i}$ from
\eqref{eq:delta-G} satisfies $\delta_G \geq \Omega(n)$.
\end{lemma}

As we observe in Section~\ref{sec:future}, this threshold is sharp in that
when the $r_i$ are proportional to $1/i$ (and thus $\E(X_r)\approx \ell/\log(\ell)$), we get $\delta_G = \Theta(n / \log
\ell)$, which is not $\Omega(n)$ unless $\ell$ is bounded.

\begin{proof}

By assumption there exists a universal constant $c>0$ such that
$\E(X_r)=(\sum ir_i)/(\sum r_i)$ is at least $\ell/c$. We will show that
$\delta_G\geq n/c$, or in other words that $(\sum_j jr_j)^2 \ge (\sum_i
r_i/c)(\sum_j j^2 r_j)$. Rewriting, our goal is to show that \ifconfversion
\else
\[\sum_{i=1}^\ell \sum_{j = 1}^\ell  (ij - j^2/c)r_ir_j \ge 0.\]
\fi 
Rewrite the sum as $\sum_j t_jr_j$, where $t_j = \sum_i (ij - j^2/c)r_i$. Since $j\leq \ell$, we have $t_j\geq \sum_i(ij-\ell j/c)r_i$, so $t_j/j\geq \sum_i ir_i - \frac{\ell}{c}\sum_i r_i$. But our assumption of linear expectation states precisely that $\sum_i i r_i\geq \frac{\ell}{c}\sum_i r_i$, so we have $t_j/j\geq 0$. We conclude that $t_j\geq 0$ and thus $\sum_j t_jr_j\geq 0$ as desired.
\end{proof}

Finally, although Corollary~\ref{cor:bounded_length} (the case of bounded
length) follows from the case of bounded variance
(Theorem~\ref{thm:pgroups}(\ref{case:bounded_var})), we record here an even
simpler proof of this corollary, which gives a more exact dependence on the
length $\ell$.

\begin{lemma}[Bounded length] \label{lem:bounded_length}
Given $r=(r_1,\ldots,r_\ell)$ with $n=\sum r_i$, \ifconfversion
\else the expression
$\delta_G=\frac{(\sum_i i r_i)^2}{\sum_i i^2 r_i}$ from $\eqref{eq:delta-G}$
satisfies $\delta_G \geq \Omega(n/\ell)$.
\fi 
\end{lemma}

This bound is tight up to a $\log^2 \ell$ factor, as can be seen for $r_i$
proportional to $1/i^2$ (see Section~\ref{sec:future}).

\begin{proof} Since
the $r_i$ are nonnegative, note that
\[
\sum_i r_i \leq \sum_i i r_i.
\]
Also, since every $i$ is at most $\ell$, we have
\[
\sum_i i^2 r_i \leq \ell \sum_i i r_i.
\]
Putting these two together, we find that
\[
\big(\sum_i r_i\big)\big(\sum_i i^2 r_i\big) \leq \big(\sum_i i r_i\big)^2 \ell.
\]
Finally, using the fact that $\sum_i r_i = n$ and rearranging, we obtain
$\delta_G \geq \Omega(n /\ell)$.
\end{proof}

\mysubsection{General nilpotent groups: extending from normal subgroups}
Recall that a finite group is nilpotent if and only if it is a direct product
of groups of prime power order, which are then its Sylow $p$-subgroups.
We say that a family of finite nilpotent groups has \emph{bounded variance}
(respectively, \emph{linear expectation}) if there is some universal constant $M$
(respectively, $c>0$) such that for each of its Sylow $p$-subgroups the $p$-degrees
have variance bounded by $M$ (respectively, have expectation at least $\ell/c$,
where $\ell$ is the length of the $p$-central series).

\begin{theorem}[Main theorem for nilpotent groups] \label{thm:nilpotent}
STPP constructions in families of nilpotent groups $G$ cannot achieve
$\omega=2$ if they have bounded exponent and either
\begin{enumerate}
\item \label{case:bounded_var_general} they have bounded variance, or
\item \label{case:linear-expectation_general} they have linear expectation.
\end{enumerate}
\end{theorem}

As in the case of $p$-groups, we have the following easily-applied corollary:

\begin{corollary}
STPP constructions in families of nilpotent groups of bounded exponent and
bounded nilpotency class cannot achieve $\omega=2$.
\end{corollary}

\begin{proof}
If $G$ has bounded nilpotency class and bounded exponent then for each of its
Sylow $p$-subgroups $P$, the $p$-lower central series has bounded length
(indeed, the length $\ell$ is at most $mc$ where $m$ is the exponent and $c$ is the
nilpotency class). So
Theorem~\ref{thm:nilpotent}(\ref{case:bounded_var_general}) applies.
\end{proof}

Our key tool for proving Theorem~\ref{thm:nilpotent} will be showing how to
extend our slice rank bounds from a normal subgroup to its parent group,
which may be of independent interest. Call a subset $J$ of a ring $R$
\emph{characteristic} if $\alpha(J) = J$ for all $\alpha \in \Aut(R)$.

\begin{lemma} \label{lem:normal}
Let $\F$ be a field, $G$ be a group, and $N \unlhd G$ be a normal subgroup.
Suppose that $I \subseteq \F[N]$ is a subspace and $J \subseteq \F[N]$ is a
characteristic right ideal. Then \[\slicerank(\F[G]) \leq
|G/N|\left(\codim_{\F[N]} I + \codim_{\F[N]} J + \dim IJ \right).\]
\end{lemma}

\ifconfversion
\else
\begin{proof}
Let $k = |G/N|$. Let $\{q_i : i \in [k]\}$ be a set of coset representatives
of $N$ in $G$. Let $I' = \bigoplus_i I \cdot q_i$ where $I \cdot q_i = \{xq_i
: x \in I\}$, and similarly let $J' = \bigoplus_i J \cdot q_i$. Then we have
$I' J' = \Span\{ x q_i y q_j : x \in I, y \in J, i,j \in [k] \}$. Since $J$
is characteristic, and conjugation by $q_i$ is an automorphism of $\F[N]$, we
have $q_i y = q_i y q_i^{-1} q_i = y' q_i$ for some other $y' \in J$. But
then we have $I' J' \subseteq \Span\{ xy q_i q_j : x \in I, y \in J, i,j \in
[k] \}$. Now, $q_i q_j = n_{ij} q_{i'}$ for some $n_{ij} \in N$ and some $i'
\in [k]$, but since $J$ is a right ideal, $y n_{ij}$ is again in $J$.
Finally, we thus have $I' J' \subseteq \Span\{ xy q_i : x \in I, y \in J, i
\in [k] \} = \bigoplus_i (IJ) \cdot q_i$. We then apply
Proposition~\ref{prop:codim}, noting that $\codim_{\F[G]} I' = |G/N|
\codim_{\F[N]}I$ and $\codim_{\F[G]} J' = |G/N| \codim_{\F[N]}J$.
\end{proof}

\begin{proof}[Proof of Theorem~\ref{thm:nilpotent}]
Let $G$ be a nilpotent group of exponent $\leq m$, and order $|G| = p_1^{n_1}
p_2^{n_2} \dotsb p_d^{n_d}$, where the $p_i$ are distinct primes. Then $G
\cong P_1 \times P_2 \times \dotsb \times P_d$ where each $P_i$ has order
$p_i^{n_i}$.  We will show that at least one of the $P_i$ satisfies $|P_i|
\geq |G|^{\Omega(1)}$. Let $k$ be the index which maximizes $n_k$. Let $N =
\sum_{i=1}^d n_i$, then $n_k \geq N/d$. Since $\sum_i n_i \ln p_i = \ln |G|$
and $p_i \leq m$ for all $i$, we have $N \geq \frac{\ln |G|}{\ln m}$. Let
$\pi(n)$ be the number of primes $\leq n$. As $\pi(n) \leq \frac{1.25506
n}{\ln n}$ for all $n$ (e.\,g., \cite{rosserSchoenfeld}), we have that $d
\leq \frac{1.3 m}{\ln m}$. Combining these bounds on $N$ and $d$, we get $n_k
\geq N/d \geq \frac{\ln |G|}{1.3 m}$. Thus $|P_k| = p_k^{n_k} \geq p_k^{\ln
|G| / 1.3 m} \geq |G|^{\Omega(1)}$ (since $m \leq O(1)$).

Now let $p = p_k$ and $P = P_k$. We consider the two cases of the theorem
separately, showing in each case that Theorem~\ref{thm:pgroups} applies to
$P$. In case (\ref{case:bounded_var_general}), by hypothesis  the $p$-degrees of $P$ have bounded
variance, and hence Theorem~\ref{thm:pgroups}(\ref{case:bounded_var}) applies
to $P$; in case (\ref{case:linear-expectation_general}), by hypothesis the $p$-degrees of $P$ have linear expectation, and hence
Theorem~\ref{thm:pgroups}(\ref{case:linear-expectation}) applies to $P$.

Finally, from the proof of Theorem~\ref{thm:pgroups} we have that there are
characteristic ideals $I,J \subseteq \F_p[P]$ (namely, certain powers of the
augmentation ideal) such that $\codim I + \codim J + \dim IJ \leq |P|^{1 -
\Omega(1)}$. Applying Lemma~\ref{lem:normal}, we get that
$\slicerank_{\F_p}(\mult{G}) \leq |G/P| |P|^{1 - \Omega(1)} = |G| /
|P|^{\Omega(1)}$, and since $|P| \geq |G|^{\Omega(1)}$ we get
$\slicerank_{\F_p}(\mult{G}) \leq |G|^{1 - \Omega(1)}$. Now apply
Corollary~\ref{cor:key}.
\end{proof}
\fi

\mysection{Ruling out constructions using Young subgroups} \label{sec:symmetric}

If one is to prove $\omega = 2$ via the group-theoretic approach, one needs
(a family of) groups $G$ with subsets $S, T, U$ that satisfy the Triple
Product Property, and with $|S|, |T|, |U|$ all at least $|G|^{1/2 - o(1)}$.
Although we conjectured in \cite{CKSU} that such constructions are obtainable
in a variety of ways in wreath product groups, there is only one {\em
currently known} construction actually achieving this bound, which appeared
in the original 2003 paper of Cohn and Umans \cite{CU03}. This is the
so-called ``triangle construction'' in the symmetric group. We recall it
here:

\begin{figure}[!htbp]
\label{fig:tri-and-hex}
\begin{center}
\begin{tikzpicture}
  \fill[black] (0,0) circle (0.075);
  \fill[black] (1,-0.866025) circle (0.075);
  \fill[black] (-1,-0.866025) circle (0.075);
  \fill[black] (0.5,0) circle (0.075);
  \fill[black] (-0.5,0) circle (0.075);
  \fill[black] (0.5,-0.866025) circle (0.075);
  \fill[black] (0,0.866025) circle (0.075);
  \fill[black] (0,-0.866025) circle (0.075);
  \fill[black] (-0.5,-0.866025) circle (0.075);
  \fill[black] (-0.75,-0.4330125) circle (0.075);
  \fill[black] (-0.25,-0.4330125) circle (0.075);
  \fill[black] (0.25,-0.4330125) circle (0.075);
  \fill[black] (0.75,-0.4330125) circle (0.075);
  \fill[black] (-0.25,0.4330125) circle (0.075);
  \fill[black] (0.25,0.4330125) circle (0.075);
\end{tikzpicture}
\hspace{.4in}
\begin{tikzpicture}
  \fill[black] (0,0) circle (0.075);
  \fill[black] (1,0) circle (0.075);
  \fill[black] (-1,0) circle (0.075);
  \fill[black] (0.5,0) circle (0.075);
  \fill[black] (-0.5,0) circle (0.075);
  \fill[black] (0.5,0.866025) circle (0.075);
  \fill[black] (0.5,-0.866025) circle (0.075);
  \fill[black] (0,0.866025) circle (0.075);
  \fill[black] (0,-0.866025) circle (0.075);
  \fill[black] (-0.5,0.866025) circle (0.075);
  \fill[black] (-0.5,-0.866025) circle (0.075);
  \fill[black] (-0.75,-0.4330125) circle (0.075);
  \fill[black] (-0.25,-0.4330125) circle (0.075);
  \fill[black] (0.25,-0.4330125) circle (0.075);
  \fill[black] (0.75,-0.4330125) circle (0.075);
  \fill[black] (-0.75,0.4330125) circle (0.075);
  \fill[black] (-0.25,0.4330125) circle (0.075);
  \fill[black] (0.25,0.4330125) circle (0.075);
  \fill[black] (0.75,0.4330125) circle (0.075);
\end{tikzpicture}
\end{center}
\caption{A triangular array of points and a hexagonal array of points.}
\end{figure}
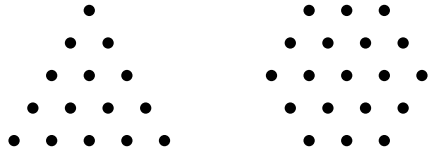

\begin{theorem}
\label{thm:triangle} Let $m$ be a positive integer and let $n = m(m+1)/2$.
Let $S_{n}$ act on the triangular array of points (as in Figure
\ref{fig:tri-and-hex}) with side length $m$. Then the three Young subgroups
$S, T, U$ that preserve lines parallel to each of the three sides,
respectively, satisfy the Triple Product Property, and $|S| = |T| = |U| =
|S_n|^{1/2 - o(1)}$.
\end{theorem}

Recall that a {\em Young subgroup} of the symmetric group $S_n$ is specified
by a partition of $[n]$, and consists of those permutations that preserve
that partition. In particular, every Young subgroups of $S_n$ is isomorphic
to $S_{n_1} \times S_{n_2} \times \dotsb \times S_{n_k}$ for some $n_i$ such
that $\sum n_i = n$.

The question of whether the triangle construction proves a nontrivial upper
bound on $\omega$ can be answered by appealing to Proposition
\ref{prop:nec-TPP}. Some algebraic manipulation and Stirling's formula shows
that \ifconfversion
\else
\[|S| = |T| = |U| \le \frac{|S_n|^{1/2}}{e^{\Omega(n)}};\]
\fi on the other hand, the number of conjugacy classes of $S_n$ is the {\em
partition number}, which is asymptotically $e^{\Theta(\sqrt{n})}$. Thus the
triangle construction in $S_n$ does not prove any non-trivial bounds on
$\omega$ because the subgroups $S, T, U$ are \emph{very slightly} too small.

However, the story does not end there. One can show that the proof of Theorem
\ref{thm:triangle} generalizes to any triple of Young subgroups that have
trivial pairwise intersection and satisfy an additional ordering axiom that
enables the inductive proof to go through. An example is the triple of Young
subgroups that preserve lines in each of the three directions parallel to the
sides of a {\em hexagon} (see Figure \ref{fig:tri-and-hex}). Intriguingly, we
find that the triangle construction is {\em not optimal}, in the sense that
for the symmetric group $S_n$, subgroups described via the hexagon are
significantly larger than subgroups described via the triangle. As a concrete
example, the hexagon with side length $6$ and the triangle with side length
$13$ both have $91$ points, yet the ratio of the size of the Young subgroups
described via the hexagon to the size of the Young subgroups described via
the triangle is
\[\frac{6!\cdot7!\cdot8!\cdot9!\cdot10!\cdot11!\cdot10!\cdot9!\cdot8!\cdot7!\cdot6!}
{13!\cdot12!\cdot11!\cdot10!\cdot9!\cdot8!\cdot7!\cdot6!\cdot5!\cdot4!\cdot3!\cdot2!\cdot1!} = 2940/1573 \ge 1.869\ldots.\]

This raises the question: could three Young subgroups (constructed via a
different ``shape'' than the triangle or perhaps not having a geometric
description at all) prove $\omega = 2$? This question is quite delicate as it
depends on the lower order terms in the size of the Young subgroups. It is
also very sensitive to these lower order terms: if one could achieve
\ifconfversion
\else
\[|S| = |T|= |U| \ge \frac{|S_n|^{1/2}}{e^{o(\sqrt{n})}},\]
\fi then this would prove $\omega = 2$ via Theorem \ref{thm:TPP}. And indeed
one {\em can} achieve this bound for two of the three subgroups by
``stretching'' the triangle in one direction. Is it possible for all three?
We prove that the answer is no, using only the fact that the three Young
subgroups must have trivial pairwise intersections (which is necessary to
satisfy the TPP):

\begin{theorem}
\label{thm:Young}
Let $H_1, H_2, H_3 \subseteq S_n$ be Young subgroups such that $H_1 \cap H_2
= H_2 \cap H_3 = H_1 \cap H_3 = \{1\}$. There exists universal constants $c,d > 0$
for which
\[\frac{|S_n|}{(|H_1||H_2||H_3|)^{2/3}} \ge e^{cn - d\sqrt{n}\log n}.\]
\end{theorem}

\ifconfversion\else
For the proof we will need the following inequality. Throughout, we will use
a version of Stirling's approximation: $e(n/e)^n \le n! \le en(n/e)^n$.

\begin{lemma}
\label{lem:binomial}
We have \[\binom{n}{t} \ge \frac{1}{e(n-t)t}\cdot\left (\frac{n}{t}\right
)^{t}\cdot e^{t(1 - t/n)}.\]
\end{lemma}

\begin{proof}
Using the above Stirling approximation, we obtain
\begin{eqnarray*}
{\binom{n}{t}} & \ge &
\frac{e(n/e)^n}{e^2(n-t)t((n-t)/e)^{n-t}(t/e)^{t}}
 = \frac{(n/t)^{t}}{e(n-t)t(1 - t/n)^{n-t}}  .
\end{eqnarray*}
Using the fact that $(1 - 1/x)^x \le 1/e$ for $x \ge 1$ completes the
proof.
\end{proof}

The proof below uses the general principle that for a Young subgroup, it is
never made smaller by transferring an element from one part in the associated
partition of $[n]$ to another part that is no smaller.

\begin{proof}[Proof of Theorem \ref{thm:Young}]
The proof is by induction on $n$. Consider the largest part of $H_1$, $H_2$
and $H_3$, and let $t$ be its size. We have three cases depending on whether
$t$ is ``large'', ``small'', or neither. The base case for the induction is
when $t = n$, which is covered by the ``large'' case below. The thresholds
for ``large'' and ``small'' are somewhat delicate. Throughout the proof we
will identify a finite number of upper bounds on the constant $c$ and a
finite number of lower bounds on the constant $d$; the final statement thus
holds for some universal constants $c, d$, which we have not explicitly
worked out to avoid cluttering the argument.

\begin{description}
\item[$t$ large] If $t > 0.9n$ (the constant $0.9$ can be replaced with
    any constant larger than $5/6$) then the subgroup containing this part
    has size at most $(n-t)!t!$, while the other two subgroups each have at
    least $t$ parts (one part for each of the elements in the size-$t$
    set), and so they have size at most $(n-t+1)!$. Thus we have
\[\frac{|S_n|}{(|H_1||H_2||H_3|)^{2/3}} \ge
\frac{n!}{((n-t+1)^2t!(n-t)!^3)^{2/3}}\ge
\frac{(n/e)^n}{(n^6(n/e)^n(0.1n/e)^{3 \cdot 0.1n})^{2/3}} =
e^{\omega(n)},\] and in particular it is at least $e^{cn - d\sqrt{n}\log
n}$ for $c \le 0.2$, and sufficiently large $d$, as required.

\item[$t$ small] If $t < e^{0.49}\sqrt{n}$ (the constant $0.49$ can be
    replaced with any constant less than $1/2$, although there is an
    interaction with the $0.9$ constant above), then for each $i$,
\[|H_i| \le (e^{0.49}\sqrt{n})!^{\sqrt{n}/e^{0.49}} \le
\left (e^{1.49}\sqrt{n}(e^{0.49}\sqrt{n}/e)^{e^{0.49}\sqrt{n}}\right
)^{\sqrt{n}/e^{0.49}} = n^{O(\sqrt{n})} n^{n/2}/e^{0.51n}\] and so
$n!/(|H_1||H_2||H_3|)^{2/3} \ge e^{0.02n - O(\sqrt{n}\log n)} $ which is at
least $e^{cn - d\sqrt{n}\log n}$ for $c \le 0.02$ and sufficiently large
$d$, as required.

\item[$t$ neither large nor small] Otherwise we have $e^{0.49}\sqrt{n} \le
    t \le 0.9n$. We consider $H_1', H_2', H_3'$, the subgroups of $S_{n-t}$
    obtained by removing the elements associated with the part of size $t$.
    We have
\begin{equation} \label{eq:induction}
\frac{|S_n|}{(|H_1||H_2||H_3|)^{2/3}} = R \cdot
\frac{|S_{n-t}|}{(|H_1'||H_2'||H_3'|)^{2/3}},
\end{equation}
where
\[R = \frac{n!}{(n-t)!(t!(a_1a_2\ldots a_t)(b_1b_2\ldots
  b_t))^{2/3}}.\] Here, if $H_1$ is the subgroup containing the part of
size $t$, then $a_1, a_2, \ldots, a_t$ are the sizes of the $t$ parts of
$H_2$ intersecting that part, and $b_1, b_2, \ldots, b_t$ are the sizes of
the $t$ parts of $H_3$ intersecting that part. Note that $\sum_i a_i \le n$
and $\sum_i b_i \le n$ and thus $\prod_i a_i \le (n/t)^t$ and $\prod_i b_i
\le (n/t)^t$.

Using this upper bound, we find that
\[R \ge \frac{{\binom{n}{t}}}{(n/t)^t}\cdot \left (
  \frac{t!}{(n/t)^{t}} \right )^{1/3} \ge \left (\frac{e^{t(1 -
    t/n)}}{e(n-t)t}\right ) \cdot \left (
  \frac{t!}{(n/t)^{t}} \right )^{1/3},\]
where the last inequality used
Lemma \ref{lem:binomial}. Now, because $t \le 0.9n$, the first factor above
is at least $e^{0.1t}/(en^2)$, and because $t \ge e^{0.49}\sqrt{n}$, the
second factor above is at least $(t^2/(en))^t \ge e^{-0.02t/3}$. We
conclude that $R \ge e^{c_1t}/n^{c_2}$ for universal constants $c_1, c_2 >
0$.

Now, by induction we know that the right hand side of \eqref{eq:induction}
is at least
\[R\cdot e^{cn' - d\sqrt{n'}\log n'},\]
where $n' = n - t \le n - e^{0.49}\sqrt{n}$. One can verify that for
sufficiently large $d$ as a function of $c_2$, it holds that
$d\sqrt{n'}\log n' \le dn\log n - c_2$. Thus, provided $c \le c_1$,
\[R \cdot e^{cn' - d\sqrt{n'}\log n'} \ge e^{cn - d\sqrt{n}\log n},\]
as required.
\end{description}
This completes the proof.
\end{proof}
\fi 

An important open question is whether this theorem can be extended to rule
out all triples of subgroups of $S_n$, or even all triples of subsets, which
would eliminate the symmetric group as a means of potentially proving $\omega
=2$. Or, does it point toward a new construction in symmetric groups that
would prove $\omega = 2$? Along these lines, it is intriguing that if we alter
the setup only slightly, the theorem fails to hold. Specifically,
the three subgroups $S = S_n \times \{1\}$, $T =
\{1\} \times S_n$, and $U = \{(\pi, \pi): \pi \in S_n\}$ of $G = S_n^2$ satisfy the
conditions of the theorem (have pairwise trivial intersection) and have
largest-possible size: $|S| = |T| = |U| = |G|^{1/2}$.

\mysection{Future directions} \label{sec:future}

This work raises as many questions as it answers, both regarding efforts to
extend the slice rank upper bounds that rule out proving $\omega = 2$ in
certain groups, and to identify groups that seem beyond the reach of that
methodology, which then may be candidates for constructions aimed at proving
$\omega = 2$.

In the negative direction, the most ambitious but realistic conjecture we
could imagine would be that {\em every} finite group $G$ of bounded exponent
has slice rank at most $|G|^{1 - \epsilon}$ for fixed $\epsilon > 0$. A
slightly less ambitious version would be to prove this for all solvable
groups of bounded exponent. A first step in this direction beyond the results
in this paper might be to either (1) extend our results to nilpotent groups
{\em without} the added condition of having bounded nilpotency class, having bounded variance, or having linear expectation, or (2) to extend to solvable
groups whose $p$-Sylow subgroups satisfy the hypothesis of Theorem
\ref{thm:pgroups}. One way to achieve (2) would be to extend
Lemma~\ref{lem:normal} to non-normal subgroups (which would have further-reaching
consequences as well). It's even open whether Lemma~\ref{lem:normal} applies
in a black-box fashion to all normal subgroups, that is, whether
$\slicerank(G) \leq \slicerank(N) |G/N|$. In general, we wonder in what
families of groups can one control the shrinkage of the powers of the
augmentation ideal.

In the positive direction, can one construct a natural example of a family of $p$-groups
whose $p$-degrees have {\em neither} bounded variance \emph{nor}
linear expectation? In such groups can one give (in increasing order of
difficulty) lower bounds on the slice rank, a construction of a large
multiplicative matching, an STPP construction, a TPP construction? For the
latter two, what bound, if any, do they prove on $\omega$? (This requires
understanding the representation theory of the $p$-group in question.)

To get a feel for groups that avoid the conditions of
Theorem~\ref{thm:pgroups}, it is perhaps a useful exercise to understand the
behavior of the scale-free quantity $\delta'_G$ from the proof of
Lemma~\ref{lem:bounded_var} for $p$-degrees $r_i$ proportional to $i^c$ for
various $c$. Using the standard estimate that $\sum_{i=1}^\ell i^c$ is
$\Theta(\ell^{c+1})$ for $c > -1$, $\Theta(\log \ell)$ for $c=-1$, and
$\Theta(1)$ for $c < -1$, we find that the scale-free quantity $\delta'_G$ of
\eqref{eq:delta-prime-G} behaves as follows:
\[
\begin{array}{rccccccc}
c: & (-\infty,-3) & -3 & (-3, -2) & -2 & (-2, -1) & -1 & (-1, \infty) \\ \hline
\delta'_G: & \Theta(1) & \Theta(1/\log \ell) & \Theta(1/\ell^{3-|c|}) & \Theta( (\log \ell)^2 / \ell) & \Theta(1/\ell^{|c|-1}) & \Theta(1/\log \ell) & \Theta(1)
\end{array}
\]
We note that the case of $ c < -3$ is a special case of bounded variance
(Lemma~\ref{lem:bounded_var}). This table shows that the condition of having linear
expectation was sharp for these families, in that  $r$ does not have linear expectation
when $c\leq -1$, and indeed we see that $\delta_G$ is \emph{not}
$\Omega(n)$ for $r_i$ proportional to $1/i^c$ with $c \in [-1, -3]$ (unless
$\ell \leq O(1)$, which was covered separately by
Lemma~\ref{lem:bounded_length}). The case of $c=-2$ (or $c \to -2$ from the
right) also shows that the bound of $\Omega(n/\ell)$
(Lemma~\ref{lem:bounded_length}) is tight. Aside from a guide to searching
for groups in which one could potentially prove $\omega=2$, is there a good
intuitive explanation for the behavior of $\delta_G$ as $c$ varies?

Regarding the symmetric group, there are intriguing questions in the
positive and negative direction as well. Here we recall that negative results
should rule out TPP constructions in $S_n$ having sets of cardinality at
least $n!^{1/2}/e^{O(\sqrt{n})}$, while achieving even slightly larger
cardinality $n!^{1/2}/e^{o(\sqrt{n})}$ proves $\omega = 2$! Incidentally,
this sharp threshold between proving no bound on $\omega$ and $\omega = 2$
occurs in any group for which $d_{\max}^2$ is polynomially related to the
average irreducible representation dimension $d_i^2$, which often makes the threshold of
$|G|/(\mbox{\# conjugacy classes})$ the ``right'' one to aim for.

Thus, in the negative direction, the most ambitious conjecture regarding the
symmetric group would be that $S_n$ has slice rank at most
$n!^{1/2}/e^{O(\sqrt{n})}$. Note that this would follow from extending Lemma
\ref{lem:normal} to non-normal subgroups, since $S_n$ has (non-normal)
$p$-subgroups of size $\exp(n)$. A smaller step in the direction of ruling
out $\omega = 2$ in the symmetric group would be to extend our result to all
triples of subgroups (not just Young subgroups).

In the positive direction, can one give in $S_n$ (in increasing order of
difficulty) lower bounds on the slice rank of $n!/e^{o(\sqrt{n})}$, a
construction of a large multiplicative matching of cardinality
$n!/e^{o(\sqrt{n})}$, a TPP construction with sets of cardinality
$n!^{1/2}/e^{o(\sqrt{n})}$? Is there a small variation on the group $S_n$ (like
the direct product of a small number of symmetric groups as suggested after
the proof of Theorem \ref{thm:Young}) that circumvents the negative results
and admits TPP constructions of the above size?

\section*{Acknowledgments}

We thank Cris Moore, with help from Aaron Clauset, for the suggestion of the
case of bounded variance; this was a case of the right name for a concept
suggesting a better theorem than we had previously. For funding
acknowledgments see the title page.

\newpage
\appendix
\mysection{Tightness of slice rank bounds}
\label{sec:tight}

In this section we show that the bounds in our main theorems, Theorem
\ref{thm:pgroups} and Theorem \ref{thm:nilpotent}, are tight.

The following is an important relaxation we need for our constructions:

\begin{definition}[{border multiplicative matchings \cite[Def.~3.2]{BCCGNSU}}]
A {\em border multiplicative matching} in $G$ is given by three
sequences of elements in $G \times \Z$,
\[((s_1, a_1), (s_2, a_2), \ldots, (s_n, a_n)),\;\;\;\; ((t_1, b_1),
(t_2, b_2), \ldots, (t_n, b_n)),\;\;\;\; ((u_1, c_1), (u_2, c_2),
\ldots, (u_n, c_n)),\]
for which the following hold:
\begin{enumerate}
\item $(s_it_ju_k = 1 \mbox{ and } a_i + b_j + c_k = 0)
  \Longleftrightarrow i = j = k$ (multiplicative matching in $G \times \Z$), and
\item $s_it_ju_k = 1 \Rightarrow a_i + b_j + c_k \ge 0$ (positivity).
\end{enumerate}
\end{definition}

We can convert a border multiplicative matching in $G$ into a multiplicative
matching in powers of $G$. This was stated in \cite{BCCGNSU} for
abelian groups, but the same proof works \emph{mutatis mutandis} for
arbitrary groups:

\begin{lemma}[{Cf.\ \cite[Lem.~3.4]{BCCGNSU}}]
\label{lem:convert-border} Suppose there exists a border multiplicative
matching in $G$ of cardinality $m$. Then for every $N$, there exists a
multiplicative matching in $G^N$ of cardinality at least
\[m^N/(2Nt+1)^3,\] where $t$ is a constant independent of $N$.
\end{lemma}

Just as with ordinary multiplicative matchings, border multiplicative
matchings are a lower bound on slice rank:

\begin{proposition}
If $G$ contains a border multiplicative matching of cardinality $m$, then the
slice rank of the $G$-multiplication tensor (in any characteristic) is at
least $m$.
\end{proposition}
\begin{proof}[Proof using powers of $G$]
This follows from the fact that slice rank is submultiplicative. Lemma
\ref{lem:convert-border} implies that there is a tricolored sum free set of
cardinality $m^N/\poly(N)$ in $G^N$, and thus the slice rank of $G^N$ is at least
$m^N/\poly(N)$. But this implies the slice rank of $G$ is at least
$(m^N/\poly(N))^{1/N}$ (by submultiplicitivity), which approaches $m$ as $N$
tends to infinity.
\end{proof}

To prove tightness of our results we need to construct large border
multiplicative matchings. Our starting point is the case of the cyclic group
$\Z/m\Z$. Here we have the following construction:

\begin{proposition} \label{prop:cyclicmatching}
There is a border multiplicative matching in $\Z/m\Z$ of
cardinality at least $\lceil m/2 \rceil$.
\end{proposition}

\begin{proof}
Identify $G = \Z/m\Z$ with $\{0,1,2, \ldots, m-1\}$ in the natural
way, and let $m' = \lceil m/2 \rceil$. Define
\[x_i = (i, i^2) \in G \times \Z,\;\;\;\; y_j = (j, j^2) \in G \times
\Z,\;\;\;\; z_k = (-2k, -2k^2) \in G \times \Z,\] where $i,j,k \in \{0,1,
\ldots, m'-1\}$. Note that $(i + j -2k, i^2 + j^2 - 2k^2) = (0,0)$ when $i =
j = k$. In the other direction, if $i + j = 2k$, then $i = k - c$ and $j = k
+ c$ for some $c \in \{-(m'-1), \ldots, (m'-1)\}$. We thus have $i^2 + j^2 -
2k^2 = 2c^2$, which is always non-negative, and equals $0$ iff $i = j = k$,
as required.
\end{proof}

Multiplicative matchings behave well when taking extensions:

\begin{lemma}
\label{lem:extension} Suppose we have a short exact sequence $1 \rightarrow N \rightarrow G
\rightarrow G/N \rightarrow 1$.  Let $\{s_i\}, \{t_j\}, \{u_k\}$ be a
multiplicative matching of cardinality $m$ in $N$ and let $\{x_{i'}\},
\{y_{j'}\}, \{z_{k'}\}$ be a multiplicative matching of cardinality $m'$ in
$G/N$. Then there is a multiplicative matching of cardinality $mm'$ in $G$.
The same statement holds with ``multiplicative matching'' replaced everywhere
by ``border multiplicative matching''.
\end{lemma}
\begin{proof}
Let $\overline{x_{i'}}$ and $\overline{y_{j'}}$ be arbitrary lifts from $G/N$
to $G$. Recall that we know $x_{k'}y_{k'}z_{k'} = 1$ in $G/N$, so we can
define $\overline{z_{k'}} = (\overline{x_{k'}}\cdot \overline{y_{k'}})^{-1}$
and still satisfy that $\overline{z_{k'}}$ mod $N$ is $z_{k'}$. We needed to
go to the extra effort of carefully defining this last lift so that
$\overline{x_{k'}}\cdot \overline{y_{k'}}\cdot \overline{z_{k'}} = 1$ in $G$,
for all $k'$.

Define $f_{i, i'} = s_i\overline{x_{i'}}$, $g_{j, j'} =
t_j^{(\overline{x_{j'}})^{-1}}\overline{y_{j'}}$ and $h_{k,k'} =
u_k^{(\overline{x_{k'}}\cdot\overline{y_{k'}})^{-1}}\overline{z_{k'}}$.
Consider the equation
\[f_{i,i'}g_{j,j'}h_{k,k'} =
s_i\overline{x_{i'}}t_j^{(\overline{x_{j'}})^{-1}}\overline{y_{j'}}u_k^{(\overline{x_{k'}}\cdot\overline{y_{k'}})^{-1}}\overline{z_{k'}}
= 1.\] Mod $N$, this reduces to the equation $x_{i'}y_{j'}z_{k'} = 1$ in
$G/N$, and so $i' = j' = k'$. The equation then simplifies to $s_it_ju_k = 1$
in $G$ and so $i = j = k$ as required.

To prove the result for border multiplicative matchings we do the same as
above and sum the second coordinates (the ones in $\Z$) when combining the
two border multiplicative matchings in $N$ and $G/N$ to obtain a border
multiplicative matching in $G$.
\end{proof}

We can now assert that our result for $p$-groups (in Theorem
\ref{thm:pgroups}) cannot be improved (up to the constant in the
$\Omega(\cdot)$ notation).

\begin{theorem}
If $G$ is a group of order $p^n$ (with $p$ prime),  then
$G$ contains a border multiplicative matching of cardinality at
least $p^n/2^n$.
\end{theorem}

\begin{proof}
Every finite $p$-group $G$ has a sequence of normal subgroups $1
\unlhd H_1 \unlhd H_2 \unlhd \dotsb \unlhd H_n = G$ such that $H_i /
H_{i-1} \cong \Z/p\Z$ for all $i$. By repeated application of the Lemma \ref{lem:extension} we find
a border multiplicative matching in $G$ of the specified size.
\end{proof}

For nilpotent groups $G$, we obtain a lower bound on {\em slice rank}
that matches our bound in Theorem \ref{thm:nilpotent}:

\begin{theorem}
If $G$ is a nilpotent group of order $p^nr$ with $(p,r)=1$, then the
slice rank of the $G$-multiplication tensor in characteristic $p$ is at
least $3 p^nr/2^{n+2}$.
\end{theorem}

\begin{proof}
Finite nilpotent groups are direct products of $p$-groups. By the previous
lemma, the $p$-Sylow factor $P$ has a border multiplicative matching of
cardinality at least $p^n/2^n$. The complement factor $H$ of order $r$ has
order relatively prime to $p$, and so the group algebra $\F_p[H]$ is
semi-simple, and thus is isomorphic to the product of matrix algebras of
dimensions $d_i$. It is known that the tensor of $d_i \times d_i$ matrix
multiplication has a border multiplicative matching of cardinality at least
$3d_i^2/4$ (see \cite{BCS}), and so the tensor of $H$-multiplication has a
border multiplicative matching (expressed in the Fourier basis) of
cardinality at least $(3/4)\cdot \sum_i d_i^2 = 3|H|/4$. These matchings in
the tensor of $\F_p[P]$-multiplication and the tensor of
$\F_p[H]$-multiplication can be combined---essentially just by tensoring them
together---to yield a border multiplicative matching in $\F_p[G]$ whose size
is the product of their sizes (similar to the proof of Lemma
\ref{lem:extension}). The theorem follows.
\end{proof}

\mysection{Other results on slice rank of groups} \label{sec:other}

Given the importance of slice rank for informing us about STPP constructions
in groups (or lack thereof), here we develop some tools and results useful
for understanding the slice rank of groups in general. In particular, we
introduce a stronger variant of slice rank that we call ``flat rank,'' which
turns out to be exactly additive. We use this to show that the slice rank of
matrix multiplication is full (i.e., $n^2$), that all semisimple group algebras
have full slice rank (equal to $|G|$)---and thus getting nontrivial upper
bounds requires working in characteristic dividing $|G|$---and finally that
every cyclic group has full slice rank in all characteristics, showing that
at least one case of Proposition~\ref{prop:augVSplcs} gave exactly the
correct value of the slice rank.

In addition to tying up some loose ends, we believe that these tools and
results may be useful for future endeavors.

\mysubsection{Flat rank} Note that, given a slice rank decomposition of a
tensor $t$ of the form
\[
\sum_{i=1}^a \alpha_i(x) \beta_i(y,z) + \sum_{j=1}^b \gamma_j(y) \delta_j(x,z) + \sum_{k=1}^c \eta_k(z) \theta_k(x,y),
\]
there is a codimension at most $c$ subspace of the $z$'s on which every $\eta_k$
vanishes. On this subspace, the resulting restricted matrix can be written as
a sum of at most $ a+b$ rank 1 matrices (just plug the $z$'s into the above
decomposition), and hence has rank at most $a+b$. Thus, in order to prove a
lower bound of $\ell$ on the slice rank of $t$, it suffices to show that
whenever $t$ is restricted to a codimension $c$ subspace of the $z$'s (for
any $c$, since now we do not imagine we know the above decomposition), the
rank of the resulting matrix is at least $\ell-c$. It turns out that this
method of proving lower bounds has several nice properties, and in particular
is exactly additive under direct sum of tensors. To formalize this, rather
than talk about ``lower bounds proven by this method,'' we capture this
method in a definition.

Given a tensor $t \colon U \otimes V \otimes W \to \F$ and a vector $w \in
W$, let $t_w$ denote the restricted tensor $t_w \colon U \otimes V \to \F$
defined by $t_w(u \otimes v) = t(u \otimes v \otimes w)$.

\begin{definition}[Flat rank]
A tensor $t \in \F^X \otimes \F^Y \otimes \F^Z$ has \emph{flat rank} at most
$r+c$ if there is a codimension $c$ subspace $V \leq \F^Z$ such that
$\Rank(t_v) \leq r$ for every $v \in V$. The flat rank of $t$, denoted
$\frank(t)$, is the minimum of $r+c$ over all such choices of subspace.
\end{definition}

As expected from the discussion preceding the definition, we have the
following relationship between flat rank and slice rank, which we restate
here for reference:

\begin{observation}
For any tensor $t$, $\frank(t) \leq \slicerank(t)$.
\end{observation}

\begin{observation}[Characterization of flat rank in terms of decompositions] \label{obs:flat_decomp}
A tensor $t$ has flat rank at most $\ell$ iff it can be written in the form
\[
\sum_{i=1}^r \alpha_i (x, z) \beta_i(y,z) + \sum_{k=1}^c \eta_k(z) \theta_k(x,y,z)
\]
with $r+c \leq \ell$.
\end{observation}

\begin{proof}
If $t$ has a decomposition as in the statement of the observation, then the
same argument as originally motivated flat rank can be used to upper bound
its flat rank, namely by restricting $t$ to the codimension at most $ c$ subspace of
$z$'s on which every $\eta_k$ vanishes, and noting that the resulting matrix
has rank at most $r$.

Conversely, suppose $\frank_z(t) = \ell' \leq \ell$. Then there exist $r,c$
such that there is a codimension $c$ subspace of the $z$'s, say $V$, with
$\max\{\Rank(t_z) : z \in V\} = r$, and $\ell' = r+c$. Let $\zeta_1, \dotsc,
\zeta_d$ be a basis of $\F^Z$ such that $\Span\{\zeta_1, \dotsc,
\zeta_{|Z|-c}\} = V$. If $Z = \{z_1, \dotsc, z_{|Z|}\}$, there is some
invertible matrix $A$ such that $\zeta_i = \sum_j A_{ij} z_j$, and it is
natural to define $t(x,y,\zeta_i) = \sum_j A_{ij} t(x,y,z_j)$. We note that
flat rank is invariant under such change of bases, so we may now reason about
$t$ in this other basis, namely via its values $t(x,y,\zeta_i)$. Then, by
assumption, for each $1 \leq i \leq d-c$, $\zeta_i \in V$, so there exist
$\alpha_{i,j}, \beta_{i,j}\colon X \to \F$ such that $t(x,y,\zeta_i) =
\sum_{j=1}^r \alpha_{i,j}(x) \beta_{i,j}(y)$. Then we have
\[
t(x,y,\zeta_i) = \sum_{j=1}^r \left(\sum_{k=1}^{d-c} \delta_{ik} \alpha_{k,j}(x)\right)\left(\sum_{\ell=1}^{d-c} \delta_{i\ell} \beta_{\ell,j}(y)\right) + \sum_{j=d-c+1}^d \delta_{ij} t(x,y,\zeta_j).
\]
Now notice that $\sum_{k=1}^{d-c} \delta_{ik} \alpha_{kj}(x)$ only depends on
$x$ and $i$, that is, $x$ and $\zeta_i$, so this whole sum may be rewritten
as some $\hat{\alpha}_j (x, \zeta_i)$, and similarly for the $\beta$'s. Next,
we may re-index the last sum to go from 1 to $c$. Finally, in the second sum,
$\delta_{ij}$ depends only on $\zeta_i$ so is of the form $\eta_j(\zeta_i)$
and $t(x,y,\zeta_j)$ doesn't depend on $\zeta_i$ at all, so it can be
rewritten as $\theta_j(x,y)$, and we are left with
\[
t(x,y,\zeta_i) = \sum_{j=1}^t \hat{\alpha}_j(x, \zeta_i) \hat{\beta}_j(y,\zeta_i) + \sum_{j=1}^c \eta_j(\zeta_i) \theta_j(x,y),
\]
which is, in fact, more restrictive than the desired form (since the
$\theta_j$ were allowed to depend on $\zeta_i$ but they don't).
\end{proof}

\begin{theorem}\label{thm:flat_additive}
Flat rank is additive. That is, given two tensors $t, t'$, $\frank(t \oplus
t') = \frank(t) + \frank(t')$.
\end{theorem}

\begin{proof}
Subadditivity---$\frank(t \oplus t') \leq \frank(t) + \frank(t')$---follows
by simply combining the representations of $t, t'$ as in
Observation~\ref{obs:flat_decomp}.

Conversely, let $r_1 = \frank(t)$ and $r_2 = \frank(t')$. Suppose that $t$ is
supported on $X \times Y \times Z$ and $t'$ is supported on $X' \times Y'
\times Z'$, with $X,X'$ disjoint, $Y,Y'$ disjoint, and $Z,Z'$ disjoint. Given
any codimension $c$ subspace $V \leq (\F^Z \oplus \F^{Z'})$, let $V|_Z$
denote the image of $V$ projected into $\F^Z$ along $\F^{Z'}$, and
analogously for $V|_{Z'}$. Let $c_Z$ denote the codimension of $V|_Z$ in
$\F^Z$, and analogously for $c_{Z'}$. Since $\dim V \leq \dim V|_Z + \dim
V|_{Z'}$, we get that $c_Z + c_{Z'} \leq c$.

Since $\frank(t) = r_1$, there must be a $z \in V|_Z$ such that $\Rank(t_z)
\geq r_1 - c_Z$.  Since having rank at most $r_1 - c_Z$ is a Zariski-closed
condition, in fact a \emph{generic} $z \in V|_Z$ satisfies $\Rank(t_z) \geq
r_1 - c_Z$; that is, the set of $z \in V_Z$ with $\Rank(t_z) < r_1 - c_Z$
is a proper subvariety, and the set of $z \in V|_Z$ with $\Rank(t_z) \geq r_1
- c_Z$ is a Zariski-open set. Its inverse image in $V$ is thus also Zariski
open. Since the analogous argument holds in $V|_{Z'}$, we find that the set
of $v \in V$ whose projection $z'$ into $\F^{Z'}$ satisfies $\Rank(t'_{z'})
\geq r_2 - c_{Z'}$ is Zariski-open in $V$. Since $V$ is an irreducible affine
variety, any two Zariski open subsets intersect, so there exists a single $v
\in V$ simultaneously satisfying these two conditions, and thus satisfying
$\Rank((t \oplus t')_v) \geq r_1 - c_Z + r_2 - c_{Z'} \geq r_1 + r_2 - c$.
Since this held for any $c$, we have that $\frank(t \oplus t') \geq r_1 + r_2
= \frank(t) + \frank(t')$.
\end{proof}

Note that the latter part of this argument depended on the field being
infinite. However, in all of our results bounding slice rank we could have
been working in an algebraically closed field without loss of generality, so
this is not a significant issue. If there ever were a setting in which one
really needed to consider slice rank or flat rank over a \emph{finite} field
$\F_q$ for which its algebraic closure $\overline{\F}_q$ wouldn't suffice, we
believe that quantitative estimates can be made for the latter half of our
argument and that the result should still hold.

\begin{corollary}
\label{cor:flat-rank-diagonal} The flat rank and slice rank of a diagonal
tensor are full.
\end{corollary}

When one applies our proof in particular to the case of a diagonal tensor,
one recovers not only Tao's result, but essentially the same proof. In this
way, our proof of additivity generalizes Tao's proof \cite{taoBlog} that a
diagonal tensor has full slice rank.

\begin{proof}
The flat rank of any $1 \times 1 \times 1$ nonzero tensor is clearly 1 (it's
not zero, and it's at most the side length). A diagonal of size $m$ is
precisely the direct sum of $m$ such tensors, so the result follows from
additivity.
\end{proof}

\begin{proposition}
The flat rank and slice rank of the square matrix multiplication tensor
$\langle n, n, n \rangle$ are full.
\end{proposition}

\begin{proof}
Suppose that $\frank(\langle n,n,n \rangle) < n^2$. So there is a codimension
$c$ subspace $V \leq M_{n \times n}$ such that the restriction of $ \langle
n,n,n \rangle$ to each vector $v \in V$ has matrix rank at most $ n^2 - c - 1$.
To see what these restricted matrices look like, let us view the matrix
multiplication in a particular form. Namely, writing out the multiplication
table in a particular order, we get the following:
\[
\begin{array}{c|cccc|ccc|c|ccc}
 & E_{11} & E_{12} & \dotsb & E_{1n} & E_{21} & \dotsb & E_{2n} & \dotsb & E_{n1} & \dotsb & E_{nn} \\ \hline
E_{11} & E_{11} & E_{12} & \dotsb & E_{1n} & 0 & \dotsb & 0 & \dotsb & 0 & \dotsb & 0 \\
E_{21} & E_{21} & E_{22} & \dotsb & E_{2n} & 0 & \dotsb & 0 & \dotsb & 0 & \dotsb & 0 \\
\vdots &  \vdots & \vdots & \ddots & \vdots & \vdots & & \vdots & & \vdots & & \vdots \\
E_{n1} & E_{n1} & E_{n2} & \dotsb & E_{nn} &0 & \dotsb & 0 & \dotsb & 0 & \dotsb & 0 \\ \hline
E_{12} & & & & & E_{11} & \dotsb & E_{1n} \\
E_{22} & & & & & E_{21} & \dotsb & E_{2n} \\
\vdots & & & & & \vdots & \ddots & \vdots \\
E_{n2} & & & & & E_{n1} & \dotsb & E_{nn} \\ \hline
\vdots & \\ \hline
E_{1n} & & & & & & & & & E_{11} & \dotsb & E_{1n} \\
\vdots  & & & & & & & & & \vdots & \ddots & \vdots \\
E_{nn} & & & & & & & & & E_{n1} & \dotsb & E_{nn} \\
\end{array}
\]
This means that the set of restrictions of $\langle n,n,n \rangle$ to $v \in
V$ is equal to
\[
\left\{
\left(\begin{array}{cccc}
X \\
 & X \\
  & & \ddots \\
  & & & X
\end{array}\right)
 :
 X \in V
\right\}.
\]
Let us denote the preceding matrix by $X^{\oplus n}$; it is clear that
$\Rank(X^{\oplus n}) = n \Rank(X)$. By assumption, we have $\Rank(X^{\oplus
n}) \leq n^2 - c - 1$ for all $X \in V$. Equivalently, $\Rank(X) \leq n -
(c/n) - (1/n)$ for all $X \in V$. However, a theorem of Flanders
\cite{flanders} says that a subspace of $n \times n$ matrices in which every
matrix has rank at most $r$ must have dimension at most $rn$. So we must have
$\dim V \leq n (n - (c/n) - (1/n)) = n^2 - c - 1$, contradicting our
assumption that $\dim V = n^2 - c$.
\end{proof}

\begin{corollary} \label{cor:semisimple}
The flat rank and slice rank of a direct sum of matrix multiplication tensors
are full. In particular, the flat rank and slice rank of any semisimple group
algebra $\F[G]$ are equal to $|G|$.
\end{corollary}

This tells us that to get any nontrivial upper bounds on
$\slicerank(\mult{G})$, one has to work over a field whose characteristic
divides $|G|$.

\mysubsection{The slice rank of cyclic groups is full} \label{sec:slice_cyclic}

\begin{theorem} \label{thm:cyclic}
For any cyclic group $G = \Z / n \Z$, $\frank(\mult{G}) = \slicerank(\mult{G}) = |G|$ in any characteristic.
\end{theorem}

While Proposition~\ref{prop:cyclicmatching} showed that our bound on the
slice rank of $\Z/{p^k}\Z$ from Proposition~\ref{prop:augVSplcs} (see
Example~\ref{ex:homocyclic}) wasn't off by more than a factor of 2,
Theorem~\ref{thm:cyclic} shows that that bound was in fact exactly correct
for cyclic $p$-groups.

\begin{lemma} \label{lem:cyclic_flat}
For $G = \Z / p^r \Z$ ($p$ prime), the flat rank of $\F_p[G]$ is equal to
$|G|$.
\end{lemma}

\begin{proof}
First, $\F_p[G] \cong \F_p[x] / (x^{p^r})$ (let $g$ be a generator of $G$;
then the element $x=g-1 \in \F_p[G]$ generates the ring, and satisfies
$x^{p^r} = 0$). The algebra multiplication table in the basis $\{1, x, x^2,
x^3, \dotsc, x^{p^r - 1}\}$ has the following form (note the order chosen on the
rows!):
\[
\begin{array}{c|ccccc}
 & 1 & x & x^2 & \dotsc & x^{p^r - 1} \\ \hline
x^{p^r -1} & x^{p^r-1} & 0 & 0 & \dotsc & 0 \\
x^{p^r - 2} & x^{p^r - 2} & x^{p^r - 1} & 0 & \dotsc & 0 \\
 \vdots & \vdots & \vdots & \vdots & & \vdots \\
 x^2 & x^2 & x^3 & x^4 & \dotsc & 0 \\
 x & x & x^2 & x^3 & \dotsc & 0 \\
 1 & 1 & x & x^2 & \dotsc & x^{p^r - 1} \\
  \end{array}
\]
In particular, each the $i$-th diagonal below the main one (with the main one
being $i=0$) has all of its entries equal to $x^{p^r - 1 - i}$. The
restriction of this tensor to any vector in the $z$-slice thus lies in the
linear span of the matrices $M_d$ defined by $M_d= \sum_i | i \rangle \langle
i + d - p^r |$ for $d=1, \dotsc, p^r$, with the understanding that $\langle m
| = 0$ if $m \leq 0$. Since such matrices are always of the form a
lower-triangular matrix surrounded by blocks of zeros, the rank of $\sum_d
a_d M_d$ is equal to the largest $d$ with $a_d \neq 0$.

Now, let $V$ be a codimension $c$ subspace of the $z$'s; so $\dim V = p^r -
c$. Then $V$ cannot be entirely contained in the span of $M_1, \dotsc,
M_{p^r-c-1}$, by dimension. Thus there is a matrix in $V$ of the form $\sum_d
a_d M_d$ with $a_d \neq 0$ for some $d \geq p^r - c$, which thus has rank at
least $p^r-c$. So the smallest upper bound possible on the flat rank is
$p^r$, and thus $\frank(\F_p[G]) = p^r = |G|$.
\end{proof}

\begin{proof}[Proof of Theorem~\ref{thm:cyclic}]
In characteristic coprime to $n$ (including characteristic zero), $\F[G]$ is
semisimple (in fact, $\F[G] \cong \F^{\oplus |G|}$), so the result follows
from Corollary~\ref{cor:semisimple}.

In characteristic $p$ with $p | n$, write $n = p^e m$ where $m$ is coprime to
$p$. Then $\Z / n \Z \cong \Z / p^e \Z \times \Z / m \Z$, so $\F[\Z / n \Z]
\cong \F[\Z / p^e \Z] \otimes \F[\Z / m \Z]$. Since $\text{char}(F)=p$ is
coprime to $m$, $\F[\Z / m \Z] \cong \F^{\oplus m}$, so we get that $\F[\Z /
n \Z]$ is isomorphic to $\F[\Z / p^e \Z]^{\oplus m}$. Since the flat rank of
$\F[\Z / p^e \Z]$ is $p^e$ (Lemma~\ref{lem:cyclic_flat}) and flat rank is
additive (Theorem~\ref{thm:flat_additive}), we get that the flat rank of
$\F[\Z / n \Z]$ is equal to $p^e m = n$.
\end{proof}

\newcommand{\etalchar}[1]{$^{#1}$}
  \newcommand{\conference}[4]{{#1} '#3: #4
  #2}\newcommand{\FOCS}[2]{\conference{FOCS}{Annual {IEEE} {S}ymposium on
  {F}oundations of {C}omputer
  {S}cience}{#1}{#2}}\newcommand{\STOC}[2]{\conference{STOC}{Annual {ACM}
  {S}ymposium on {T}heory of
  {C}omputing}{#1}{#2}}\newcommand{\MFCS}[2]{\conference{MFCS}{{S}ymposium on
  {M}athematical {F}oundations of {C}omputer
  {S}cience}{#1}{#2}}\newcommand{\SODA}[2]{\conference{SODA}{{ACM}--{SIAM}
  {S}ymposium on {D}iscrete
  {A}lgorithms}{#1}{#2}}\newcommand{\ICALP}[2]{\conference{ICALP}{{I}nternational
  {C}olloquium on {A}utomata, {L}anguages and
  {P}rogramming}{#1}{#2}}\newcommand{\CCC}[2]{\conference{CCC}{{IEEE}
  {C}onference on {C}omputational
  {C}omplexity}{#1}{#2}}\newcommand{\STACS}[2]{\conference{STACS}{Annual
  {S}ymposium on {T}heoretical {A}spects of {C}omputer
  {S}cience}{#1}{#2}}\newcommand{\journal}[1]{preliminary version of
  \cite{#1}}\newcommand{\original}[1]{originally appeared as
  \cite{#1}}\newcommand{\available}[1]{also available as
  \cite{#1}}\newcommand\arXiv[1]{arXiv:\href{http://arXiv.org/abs/#1}{#1}}\newcommand{\doi}[1]{\href{http://dx.doi.org/#1}{doi:#1}}

\end{document}